\documentclass[11pt]{article}

\usepackage{amssymb,amsmath,amsthm,amsfonts,commath,mathrsfs,mathtools,secdot}
\usepackage[T1]{fontenc}
\usepackage[a4paper, margin=1in]{geometry}

\usepackage{bbm}
\usepackage{extarrows}
\usepackage{enumerate}
\usepackage{multicol}
\usepackage{fge}
\usepackage{oldgerm}
\usepackage{hyperref}
\usepackage{paralist}
\usepackage{accents}

\numberwithin{equation}{section}
\usepackage{tikz}
\usetikzlibrary{datavisualization,datavisualization.formats.functions}
\usepackage{pgfplots}
\usepackage{tkz-fct}
\usepackage{enumitem}
\setlist[enumerate]{label=(\arabic*)}
\pgfplotsset{
  every axis/.append style={
    axis x line=middle,
    axis y line=middle,
    axis line style={<->},
    xlabel={$x$},
    ylabel={$y$},
  },
  cmhplot/.style={color=blue,mark=none,line width=1pt,<->},
  soldot/.style={color=blue,only marks,mark=*},
  holdot/.style={color=blue,fill=white,only marks,mark=*},
}
\tikzset{>=stealth}
\sectiondot{subsection}

\usepackage{xcolor}

\newtheoremstyle{assump}%
  {5pt}
  {5pt}
  {\itshape}
  {}
  {\bfseries}
  {.}
  {0.5em}
  {}

\theoremstyle{plain}
\newtheorem{theorem}{Theorem}[section]
\newtheorem{lemma}[theorem]{Lemma}
\newtheorem{prop}[theorem]{Proposition}
\newtheorem{corollary}[theorem]{Corollary}

\theoremstyle{definition}
\newtheorem{definition}[theorem]{Definition}

\theoremstyle{remark}
\newtheorem{remark}{Remark}[section]

\theoremstyle{assump}
\newtheorem{assumption}{Assumption}[section]

\begin{document}
\title{The mean-field control problem for heterogeneous forward-backward systems
}
\author{Andreas S{\o}jmark and Zeng Zhang}
\date{}
\maketitle

\vspace{-9pt}
	
	\begin{abstract}
We study the problem of mean-field control when the state dynamics are given by general systems of forward-backward stochastic differential equations (FBSDEs) with heterogeneous mean-field interactions. Firstly, we introduce a novel methodology for reducing the well-posedness of such systems to that of a single randomized mean-field FBSDE. As a consequence, we show that, in the fully coupled case, smallness conditions yield existence and uniqueness for both the system itself and the associated variational and adjoint systems. Secondly, we derive a stochastic maximum principle and a verification theorem for the mean-field control problem. This gives necessary and sufficient conditions for optimality.
	\end{abstract}
 \vspace{1pt}

\section{Introduction}

The theory of mean-field control provides a natural framework for problems in which a central decision-maker seeks to minimise an aggregate cost over a large system of state processes. Clasically, it is assumed that the system is homogeneous and interacting in an exchangeable manner, so that it can be described by a representative state equation of McKean--Vlasov, or mean-field, type. This setting has been studied extensively, and we refer to~\cite{CarmonaDelarue2018} for a comprehensive overview of treatments via both dynamic programming and the stochastic maximum principle.

The latter has led to novel well-posedness theories for FBSDEs of mean-field type, since the optimality conditions in the Pontryagin maximum principle yield a BSDE as the adjoint equation for the controlled SDE; see \cite{BYZ2015, BuckdahnDjehicheLi2011, BuckdahnLiMa2016, CarmonaDelarue2015}. The case where also the joint law with the control affects the state and costs, as it will herein, was treated in \cite{AcciaioBackhoffCarmona2019}.

Leaving aside the mean-field aspect, there has been a growing interest in problems where the state processes themselves evolve as FBSDEs. This is a natural approach to many models in finance and economics, due to the central idea that expectations about future values (of some economic variables, say) feed back into their present evolution. One of the key early examples is the FBSDE formulation of Black's conjecture for the coupled evolution of a short rate and the associated consol rate \cite{Duffie1995}. Moreover, \cite{Kartala2020, Yannacopoulos2008} have shown how the state dynamics in a range of classical problems, such as exchange rate determination and the interaction of output with stock markets, may be recast in terms of FBSDEs, and the same has recently been observed in work on dynamic contagion \cite{JettkantSojmark2025}. When introducing a policy instrument in such models, one naturally arrives at control problems for FBSDEs, as in the recent work of \cite{HuJiXue2023}.

Another route to FBSDE dynamics in finance and economics arises when forward wealth or price dynamics are coupled with BSDE representations of related derivative prices~\cite{ElKaroui1997}, recursive utilities~\cite{Duffie1992}, or ambiguity-adjusted returns~\cite{Chen2002}. Similarly to \cite{HuJiXue2023}, this motivates the recent analysis of general classes of control problems for FBSDEs in \cite{WangYongZhou2024}.

In this work, we move from the optimal control of FBSDEs to the problem of mean-field control in large heterogeneous systems of FBSDEs (e.g., representing agents or economic variables) which interact through aggregate quantities related to the law of each state process and its control. This can, e.g., include a continuum of state equations that each depend on the individual laws of the other states through a graphon \cite{Lovasz2012} or related kernels. With regards to finance and economics, heterogeneous mean-field interactions of this type have been utilized in the study of systemic risk \cite{FeinsteinSojmark2021, NeumanTuschmann2024}, portfolio optimization \cite{TangpiZhou2024}, and market equilibria \cite{WeberSojmarkLi2026}. 

As far as we are aware, there are no existing results on the mean-field control problem for heterogeneous forward-backward systems. Our central contribution is to fill this gap by developing a stochastic maximum principle for a general formulation of the problem. Moreover, our approach to dealing with the heterogeneity can be of interest, on its own, for the analysis of heterogeneous mean-field systems more generally. At the cost of a little abstraction in the background, we provide a concise and self-contained methodology that can simplify the analysis and yield more general results compared to existing approaches.

\subsection{The mean-field control problem}

For a given time horizon $T>0$, we consider a system of forward-backward stochastic differential equations (FBSDEs) on $[0,T]$, where the forward and backward states, $X^u$ and $Y^u$, are indexed by $u\in U$, for a given index set $U$ which could be countable or a continuum. The dynamics of each pair $(X^u,Y^u)$ is driven by a $d$-dimensional Brownian motion $B^u$ with $(B^u)_{u\in U}$ forming a family of independent Brownian motions. The system then takes the form
\begin{equation}
\begin{cases}\label{eq:first_FBSDE_system}
&\mathrm{d}X_t^u = b^u\bigl(t,X_t^u, Y_t^u, Z^u_t, \alpha^u_t, (\mathbb{P}^{\tilde{u}}_{t})_{\tilde{u}\in U}\bigr)\mathrm{d}t+ \sigma^u\bigl(t,X_t^u, Y_t^u, Z^u_t, \alpha^u_t, (\mathbb{P}^{\tilde{u}}_{t})_{\tilde{u}\in U} \bigr)\mathrm{d}B^u_t \\
&-\mathrm{d}Y^u_t = f^u\bigl(t,X_t^u, Y_t^u, Z^u_t, \alpha^u_t, (\mathbb{P}^{\tilde{u}}_{t})_{\tilde{u}\in U}\bigr)\mathrm{d}t -Z^u_t \mathrm{d}B^u_t,\\
\end{cases}
\end{equation}
for $u\in U$, with initial conditions $X_0^u = \chi_0^u$ and terminal conditions
$Y^u_T = G^u\bigl(X^u_T,(\mathbb{P}^{\tilde{u}}_{T,1})_{\tilde{u}\in U}\bigr)$, where we write
\begin{equation}\label{eq:mf_condition}
\mathbb{P}^{u}_{t} = \mathrm{Law}(X_t^{u},Y_t^{u},Z_t^{u},\alpha^{u}_t) \quad \text{and} \quad \mathbb{P}^{u}_{T,1} = \mathrm{Law}(X^{u}_T),
\end{equation}
for $u\in U$ and $t\in[0,T]$. The processes $X^u$, $Y^u$, $Z^u$, and $\alpha^u$ take values in $\mathbb{R}^n$, $\mathbb{R}^l$, $\mathbb{R}^{l \times d}$, and $\mathbb{R}^k$ respectively. We note that the notation  $\mathbb{P}^{u}_{T,1}$ refers to the first marginal of $\mathbb{P}^{u}_{T}$.

In the above formulation, each process $\alpha^u$ will act as a control, which is chosen by a central decision-maker in order to assert a desired influence on the corresponding state dynamics $(X^u,Y^u)$. Importantly, these dynamics can differ across the indices $u \in U$, and the whole system is coupled through a form of heterogeneous mean-field interaction via the joint laws \eqref{eq:mf_condition}.

The different indices $u\in U$ are meant to capture different characteristics within a large population of agents, a large cloud of particles, or a large set of economics variables, say. We let the distribution of these characteristics across the index space $U$ be given by a probability measure $m$ on the Borel $\sigma$-algebra $\mathcal{B}(U)$, where $U$ is assumed to be a Polish space.

The mean-field control problem that we are interested in is as follows. In view of the controlled FBSDE dynamics for the state processes in \eqref{eq:first_FBSDE_system}, the central decision-maker seeks to minimise an aggreagate cost functional of the form
\begin{equation}\label{eq:cost_functional}
J(\boldsymbol{\alpha}) =\int_U \mathbb{E}  \Bigl[  \int_0^T \!\!\ell^u\bigl(t,X_t^u, Y_t^u, Z^u_t, \alpha^u_t, (\mathbb{P}^{\tilde{u}}_{t})_{\tilde{u}}\bigr)\mathrm{d}t 
+ h^u\bigl(X^u_T,(\mathbb{P}^{\tilde{u}}_{T,1})_{\tilde{u}}\bigr) + g^u(Y^u_0)   \Bigr]\mathrm{d}m(u),
\end{equation}
by choosing an optimal control $\boldsymbol{\alpha}=(\alpha^u)_{u\in U}$ from a set of admissible controls $\mathcal{A}_{\text{ad}}$.

Regarding the heterogeneous mean-field interactions, a typical graphon style example would be that the system is coupled through aggregate quantities of the form
\[
\int_U \kappa(u,\tilde{u}) \mathbb{E}[\phi(X^{\tilde{u}}_t, Y^{\tilde{u}}_t, Z^{\tilde{u}}_t, \alpha^{\tilde{u}}_t)]\mathrm{d}m(\tilde{u}),\quad\text{for}\quad  u\in U,
\]
where $\kappa : U 
\times U \rightarrow \mathbb{R}$ is given kernel, which prescribes a directed network structure on $U$, and $\phi : \mathbb{R}^n \times \mathbb{R}^l\times \mathbb{R}^{l \times d} \times \mathbb{R}^k \rightarrow \mathbb{R}$ is some Lipschitz continuous function. However, the dependence on the laws $(\mathbb{P}^{\tilde{u}}_{t})_{\tilde{u}\in U}$ can be more general, as long as suitable Lipschitz conditions involving the Wasserstein distance are satisfied. See Section \ref{sect:FBSDE_system} for the precise assumptions.

\subsection{Related literature and main contributions}\label{sect:literature}

The mean-field control problem for fully coupled FBSDEs in the classical mean-field setting has been studied in \cite{CHEN2023105550}. Specifically, that work develops a stochastic maximum principle for the problem corresponding to \eqref{eq:first_FBSDE_system}--\eqref{eq:cost_functional} with a single index $U=\{0\}$ and $m=\delta_{0}$, and with the associated joint law $\mathbb{P}^0_t$ replaced by dependence on the marginal laws $\mathrm{Law}(X_t^0)$, $\mathrm{Law}(Y_t^0)$, $\mathrm{Law}(Z_t^0)$, and $\mathrm{Law}(\alpha_t^0)$. The latter comes from the fact that the Lipschitz assumptions in \cite{CHEN2023105550} are formulated in terms of the Wasserstein distance for each marginal.

Like \cite{CHEN2023105550}, we base our structural assumptions on the monotonicity and Lipschitz conditions for FBSDEs from \cite{Pardoux} which involve a smallness condition on how the forward part depends on the backward part. However, we modify these to allow for heterogeneous mean-field interactions via additional Lipschitz conditions for a natural variant of the Wasserstein distance (see \eqref{eq:wasserstein}) which treats the family of laws \eqref{eq:mf_condition} as a Markov kernel with respect to $(\mathcal{U},m)$. Moreover, we work with the full $\mathrm{Law}(X_t^u,Y_t^u,Z_t^u,\alpha^u_t)$ rather than restricting to marginals. By looking at convex perturbations of the optimal control, we derive our maximal principle in a similar spirit to \cite{CHEN2023105550}, but we face new challenges in dealing with (i) heterogeneous mean-field interactions, and (ii) a system driven by a continuum of independent Brownian motions.

The maximum principle for the classical mean-field control problem goes back to \cite{andersson2011}. While working on the present paper, two new preprints have appeared that, independently, develop corresponding maximum principles for the case of heterogeneous mean-field interactions \cite{CaoLauriere2025, KharroubiMekkaouiPham2025}. Both works consider systems of SDEs driven by a continuum of independent Brownian motions with coefficients that depend on the given state $X_t^u$, the control $\alpha^u_t$, and the laws $(\mathrm{Law}(X_t^{\tilde{u}}))_{\tilde{u}\in U}$, where we use our notation for the indexing. The costs are analogous to \eqref{eq:cost_functional} without the $g^u(Y_0^u)$ term and with $\ell^u$ depending only on $X_t^u$, $\alpha^u_t$, and $(\mathrm{Law}(X_t^{\tilde{u}}))_{\tilde{u}\in U}$.

The specific graphon setting of \cite{CaoLauriere2025} allows them to rely on the standard notion of Lions differentiability, while \cite{KharroubiMekkaouiPham2025} considers more general dependence on the family of laws, as in our formulation, which leads them to rely on a particular notion of linear functional derivatives from the antecedent work \cite{CoppiniDeCrescenzoPham2025} (treating the dynamic programming principle). We provide novel insights even for forward-only dynamics, as our maximal principle covers general interaction via the joint laws $(\mathrm{Law}(X_t^{\tilde{u}},\alpha^{\tilde{u}}_t))_{\tilde{u}\in U}$ and we deal with the heterogeneous mean-field aspect in a new way: by doing an additional `lift' in accordance with how we treat the family of laws as a Markov kernel for $(\mathcal{U},m)$, we automatically arrive at a suitable heterogeneous variant of the Lions derivative directly from the standard machinery of Lions differentiability.

The latter is closely related to how we approach well-posedness, including measurability in the index variable $u$ which is another key element of the analysis in \cite{CaoLauriere2025, KharroubiMekkaouiPham2025} and earlier works on heterogeneous systems \cite{BCW2023, BWZ2023, LackerSoret2023}. By setting up a suitable background space and combining disintegration results with generalised Yamada--Watanabe arguments, we develop a simple and effective framework for reducing the well-posedness of the system \eqref{eq:first_FBSDE_system}--\eqref{eq:mf_condition} to the well-posedness of a single randomised mean-field FBSDE. This works directly with the solution, rather than relying on various approximations, and could be applied more generally. In particular, it overcomes some limitations imposed by the measurability arguments in \cite{CaoLauriere2025, KharroubiMekkaouiPham2025}.

\subsection{Structure of the paper}
Section \ref{section 2} fixes notation and introduces the spaces we work with. Section \ref{sect:fbsde_system_defn} gives our definition of a solution to the FBSDE system (Definition \ref{def:strong_soln}) and some properties. Section \ref{section 2.2} presents our well-posedness results (Theorem \ref{Theorem 4.2}). Section \ref{section 4} starts by specifying the mean-field control problem. Section \ref{Lion construct} presents a suitable variant of Lions differentiability. Section \ref{sect:max_principle} derives the variational and adjoint equations, establishing their well-posedness for a given admissible control (Theorem \ref{thm:FBSDE_variational_adjoint}), and then provides the maximum principle (Theorem \ref{Tm:maximum principle}). Section \ref{sect:verification} gives the verification theorem (Theorem \ref{thm:verification}). To get directly to the maximum principle, we defer the proof of well-posedness (in the sense of Definition \ref{def:strong_soln}) to Section \ref{sect:well-posedness}.

\section{Well-posedness of the mean-field FBSDE system}\label{sect:FBSDE_system}

To analyse the control problem discussed in the introduction, we must first address the well-posedness of the class of mean-field problems we are considering. 

\subsection{Preliminaries and notation}\label{section 2}
We let $U$ be a Polish space and fix a probability measure $m$ on $\mathcal{B}(U)$. We denote by $\mathcal{U}$ the completion of $\mathcal{B}(U)$ with respect to $m$, and we continue to write $m$ for its extension to $\mathcal{U}$. Next, we let $(\Omega , \mathcal{F},\mathbb{P})$ be a complete probability space, which we take to support a family of independent $d$-dimensional Brownian motions $(B^u)_{u}$ along with a family of initial random variables $(\chi_0^u)_{u}$ that are independent of $(B^u)_{u}$. Here, and throughout, we use the shorthand notation $(\cdot)_{u}$ in place of $(\cdot)_{u\in U}$. For $u\in U$, we let $\mathbb{F}^u = \{\mathcal{F}^u_t\}_{t \in [0,T]}$ denote the filtration generated by $B^u$ and $\chi_0^u$, augmented by the $\mathbb{P}$-null sets of $\mathcal{F}$.

We shall be working with families of probability measures indexed by $u\in U$. For a given Polish space $E$, we thus define the space
\begin{equation}\label{eq:wasserstein}
\mathcal{P}_m^2(E)  = \{(\mu^u)_{u } : \mu^u\in \mathcal{P}^2(E),\;(\mu^u)_{u} \;\text{is a Markov kernel,} \; \int_U \int_E x^2 \mathrm{d}\mu^u(x) \mathrm{d}m(u) < \infty \}.
\end{equation}
Here, we mean that $(\mu^u)_{u }$ is a Markov kernel for $\mathcal{U}$. That is, for every $A\in \mathcal{B}(E)$, the map $u\mapsto \mu^u(A)$ is $\mathcal{U}$-measurable, and, of course, $A \mapsto \mu^u(A)$ is a probability measure on $\mathcal{B}(E)$, for each $u\in U$. For $\mu, \nu \in \mathcal{P}_m^2(E) $, it is natural to introduce the distance
\[W_{2,m}^2 (\mu,\nu) := \int_U W^2_2 (\mu^u,\nu^u) \mathrm{d}m(u),
\]
where $W_2$ is the usual Wasserstein $2$-distance on $\mathcal{P}^2(E)$. This distance is well-defined by the Markov kernel property of elements in $\mathcal{P}_m^2(E)$.

We let $L^2_{m}(\mathbb{R}^n)$ denote the space of all families of $\mathbb{R}^n$-valued random variables $\eta = (\eta^u)_{u}$ such that $(\mathrm{Law}(\eta^u))_{u}$ is a Markov kernel with
\[\mathbb{E}\big[|\eta^u|^2\big]   < \infty \quad \forall u \in U,  \quad  \text{and} \quad
\int_U \mathbb{E}\big[|\eta^u|^2\big]  \mathrm{d}m(u) < \infty.
\]
Furthermore, we let $\mathcal{M}^2_{\mathbb{F},m}(\mathbb{R}^n)$ be the space of all families of processes $X = (X^u)_u$, such that (i) for all $u \in U$, $X^u$ is $\mathbb{F}^u$-progressively measurable with values in $\mathbb{R}^n$, and (ii) $(\text{Law}(X^u))_{u}$ is a Markov kernel with
$$
\int_0^T\mathbb{E}  [  |X_t^u|^2   ] \mathrm{d}t  < \infty \quad\forall u \in U,  \quad  \text{and} \quad
\int_U \int_0^T\mathbb{E}  [  |X_t^u|^2   ]\mathrm{d}t \mathrm{d}m(u) < \infty.$$
Finally, we denote by
$$ \mathcal{S}^2_{\mathbb{F},m}\times \mathcal{H}^2_{\mathbb{F},m} = \mathcal{S}^2_{\mathbb{F},m} (X,Y)\times \mathcal{H}^2_{\mathbb{F},m}(Z,\Lambda)$$ 
the space of all families of processes $(X^u,Y^u,Z^u,\Lambda^u)_u$ such that (i) each $(X^u,Y^u,Z^u,\Lambda^u)$ is $\mathbb{F}^u$-progressively measurable, 
(ii) for every $u \in U$,
$$ \mathbb{E}  [ \sup_{t \in [0, T]} |X^u_t|^2  ]+\mathbb{E}  [ \sup_{t \in [0, T]} |Y^u_t|^2  ]+ \int_0^T\mathbb{E}  [  |Z_t^u|^2   ] \mathrm{d}t+ \int_0^T\mathbb{E}  [  |\Lambda_t^u|^2   ] \mathrm{d}t< \infty,$$
and (iii) we have that $(\mathrm{Law}(X^u,Y^u,Z^u,\Lambda^u))_u$ is a Markov kernel with 
$$ \int_U \Bigl( \mathbb{E}  [ \sup_{t \in [0, T]} |X^u_t|^2  ]+\mathbb{E}  [ \sup_{t \in [0, T]} |Y^u_t|^2  ]+ \int_0^T\mathbb{E}  [  |Z_t^u|^2   ] \mathrm{d}t+ \int_0^T\mathbb{E}  [  |\Lambda_t^u|^2   ] \mathrm{d}t \Bigr) \mathrm{d}m(u)< \infty.$$ 
To simplify the notation, we will generally omit $\mathbb{F}$ in the subscript of the above spaces, as we work with a given filtration. Also, we shall suppress the domain of integration $U$ when integrating against $m$, whenever the context makes clear that the integration is over the entire domain.

\subsection{The FBSDE system}\label{sect:fbsde_system_defn}
 
Let
$G : U \times \mathbb{R}^n \times \mathcal{P}^2_m(\mathbb{R}^n ) \to \mathbb{R}^l$ and $\Lambda: U \times [0,T] \times \mathbb{R}^n \times C([0,T], \mathbb{R}^d)\rightarrow \mathbb{R}^k$ be given. We assume these are Borel measurable functions, and we express them in the form $(u,x,\zeta)\mapsto G^u(x,\zeta)$ and $(u,t,x,w) \mapsto\Lambda^u(t, x, w)$. Furthermore, we assume that  $\int_0^T \mathbb{E}[|\Lambda^u(t,\chi_0^u,B^u_{\cdot\land t})|^2]\text{d}t<\infty$, for every $u\in U$, and $\int\int_0^T \mathbb{E}[|\Lambda^u(t,\chi_0^u,B^u_{\cdot\land t})|^2]\text{d}t\text{d}m(u)<\infty$.

We are then looking for $(X^u,Y^u,Z^u,\Lambda^u)_u\in \mathcal{S}^2_m \times \mathcal{H}^2_m$ satisfying the system of FBSDEs
\begin{equation} \label{eq:3.2}
   \begin{cases}
&\mathrm{d}X^u_t = b^u\bigl(t,X_t^u, Y_t^u, Z^u_t,\Lambda_t^u, (\mathbb{P}^{\tilde{u}}_{t})_{\tilde{u}}\bigr)\mathrm{d}t+ \sigma^u\bigl(t,X_t^u, Y_t^u, Z^u_t,\Lambda^u_t, (\mathbb{P}^{\tilde{u}}_{t})_{\tilde{u}}\bigr)\mathrm{d}B^u_t \\
&-\mathrm{d}Y^u_t = f^u\bigl(t,X_t^u, Y_t^u, Z^u_t,\Lambda^u_t, (\mathbb{P}^{\tilde{u}}_{t})_{\tilde{u}}\bigr)\mathrm{d}t -Z^u_t \mathrm{d}B^u_t,\\
\end{cases}   
\end{equation}
for $u\in U$, with initial conditions $X_0^u = \chi_0^u$ and terminal conditions
$Y^u_T = G^u(X^u_T,( \mathbb{P}^{\tilde{u}}_{T,1})_{\tilde{u}})$, where $(\Lambda^u)_u$ satisfies $\int_0^T \mathbb{E}  [|\Lambda_t^u-\Lambda^u(t,\chi^u_0,B^u_{\cdot \land t})|^2 ]\mathrm{d}t = 0 $, for all $u\in U$, and where
\[
\mathbb{P}^{u}_{t} = \mathrm{Law}(X_t^u,Y_t^u,Z_t^{u},\Lambda^{u}_t),\]
for all $u\in U$ and $t\in[0,T]$.
As above, $\mathbb{P}^{u}_{T,1}$ is the first marginal of $\mathbb{P}^{u}_{T}$, i.e.,  $\mathbb{P}^{u}_{T,1}=\mathrm{Law}(X^{u}_T)$.

\begin{definition}[Strong solution]\label{def:strong_soln} Let $(X^u,Y^u,Z^u,\Lambda^u)_u$ in $\mathcal{S}^2_m \times \mathcal{H}^2_m$ satisfy the system of FBSDEs \eqref{eq:3.2} subject to the stated conditions, and let $\mathbb{P}^{u} = \mathrm{Law}(X^{u},Y^{u},Z^u,\Lambda^u)$ define a Markov kernel for the completion $\mathcal{U}$ of $\mathcal{B}(U)$ in the sense that there exists $U_0\in \mathcal{B}(U)$ with $m(U_0)=1$ such that $(\mathbb{P}^{u})_{u\in U_0}$ is Markov kernel for $\mathcal{B}(U_0)$. Moreover, suppose that
\begin{equation}\label{eq:project_property}
\mathbb{P}^{u}_t=\mathbb{P}\circ(X^{u}_t,Y^{u}_t,Z^u_t,\Lambda^u_t)^{-1}=\mathbb{P}^{u}\circ(x_t,y_t,z_t,\lambda_t)^{-1},
\end{equation}
for all $t\in[0,T]$, where $\mathbb{P}^{u}$ is realised on a particular path space $S_1$ and $(x_\cdot,y_\cdot,z_\cdot,\lambda_\cdot)$ is the canonical process on that space, as defined in \eqref{eq:S1S2} and \eqref{eq:canonical_process} respectively. Then, we say that $(X^u,Y^u,Z^u,\Lambda^u)_u$ is a strong solution to the mean-field FBSDE system \eqref{eq:3.2}.
\end{definition}

\begin{remark} For details on the path space, we refer to the beginning of Section \ref{sect:well-posedness}, where all the relevant machinery is introduced before establishing well-posedness. We also stress that the dynamics of \eqref{eq:3.2} and its terminal conditions are not affected by changing $(\mathbb{P}^u)_u$ on an $m$-null set. So, from that point of view, one could also just talk about a solution to \eqref{eq:3.2} as only having to be indexed by $u\in U_0$ for some set $U_0\in\mathcal{B}(U)$ with $m(U_0)=1$.
\end{remark}

Trivially, $t\mapsto \mathbb{P}\circ(X^{u}_t,Y^{u}_t)^{-1}$ is continuous for the topology of weak convergence of measures, but this need not be the case for $\mathbb{P}^{u}_t$. Thus, joint measurability in $(t,u)$ is not immediate, but we have it by our definition of a solution, as confirmed by the following lemma.

\begin{lemma}[Joint measurability]Let $\phi:U\times[0,T]\times \mathbb{R}^n \times \mathbb{R}^l \times \mathbb{R}^{l\times d}\times \mathbb{R}^k \rightarrow [0,\infty)$ be Borel measurable. Then 
\[
(t,u)\mapsto \langle \mathbb{P}^u_t ,\phi \rangle = \mathbb{E}[\phi^u( t,X_t^u,Y_t^u,Z_t^u,\Lambda_t^u)]
\]
is $\mathrm{Leb}\otimes\mathcal{U}$-measurable.
\end{lemma} \label{lm:joint_measureable}
\begin{proof}
By construction of the canonical process in Section \ref{sect:well-posedness}, see \eqref{eq:canonical_process}, we have that \[\bigl(u,t,(x,y,[z],[\lambda])\bigr)\mapsto \phi^u(t,x_t,y_t,z_t,\lambda_t)\] is Borel measurable on $U\times [0,T]\times S_1$. Next, \eqref{eq:project_property} gives
\[
\langle \mathbb{P}^u_t ,\phi \rangle = \int_{U\times [0,T]\times S_1} \phi^u(t,x_t,y_t,z_t,\lambda_t)\text{d}\mathbb{P}^u(x,y,[z],[\lambda]).
\]
Thus, the conclusion follows from \cite[Lemma 3.2(i)]{kallenberg}, as $(\mathbb{P}^u)_u$ is a Markov kernel for $\mathcal{U}$.
\end{proof}

This lemma confirms that we can integrate expectations in time and index, and that we can rely on Fubini's theorem when doing so. Naturally, the conclusion also applies if we include $\chi_0^u$ and $B^u$ in $\phi$ (technically, it is not part of the definition, but we note that the Markov kernel property of the joint law of these and the solution is considered in the proof of Theorem \ref{prop:Big-to-smal_SDE}). Thus, by Fubini's theorem, the condition $ \int_0^T \mathbb{E}  [|\Lambda_t^u-\Lambda^u(t,\chi^u_0,B^u_{\cdot \land t})|^2 ]\mathrm{d}t = 0 $ for all $u \in U$ gives
\[
\int_U \mathbb{E}[\phi^u(X_t^u,Y_t^u,Z_t^u,\Lambda_t^u)]\text{d}m(u)=\int_U \mathbb{E}\bigl[\phi^u\bigl(X_t^u,Y_t^u,Z_t^u,\Lambda^u(t,\chi_0^u,B^u_{\cdot \land t})\bigr)\bigr]\text{d}m(u),
\]
for a.e.~$t\in[0,T]$, and of course it is measurable in $t$. This confirms that we can safely take $\Lambda^u$ in \eqref{eq:3.2} to mean the $\mathbb{F}^{u}$-progressive process $\Lambda^u(t,\chi_0^u,B^u_{\cdot \land t})$, as they are the same from the point of view of the mean-field terms and, in turn, the FBSDE dynamics (using also the aforementioned condition again). The additional flexibility is purely technical in nature, in order for things to work seamlessly when realising the laws $\mathbb{P}^u$ on the appropriate path space.

 \subsection{Well-posedness}\label{section 2.2}

To establish well-posedness, we impose the following structural conditions.

\begin{assumption} \label{Assumption 4.1} For any $\eta \in \mathcal{P}^2_m(E)$, where $
E = \mathbb{R}^n \times \mathbb{R}^l \times \mathbb{R}^{l \times d} \times \mathbb{R}^k,$
we denote by $\eta^{(i)}$ the $i$-th marginal of $\eta$, for $i = 1,2,3,4$. We use the notation \( x, x_1, x_2 \in \mathbb{R}^n \), \( y, y_1, y_2 \in \mathbb{R}^l \), \( z, z_1, z_2 \in \mathbb{R}^{l \times d} \), $ r \in \mathbb{R}^k$, $\zeta, \zeta_1, \zeta_2 \in \mathcal{P}^2_m(\mathbb{R}^n)$, and $\xi, \eta_1, \eta_2 \in \mathcal{P}^2_m(E)$ where we have $\eta^{(4)}_1= \eta^{(4)}_2$. The following conditions must hold for all such variables: 

\begin{enumerate}
    \item The functions $b, \sigma, f: U \times [0,T] \times E \times \mathcal{P}^2_m(E) \rightarrow \mathbb{R}^n , \mathbb{R}^{n \times d}, \mathbb{R}^l$ are Borel measurable. Moreover,  we assume the initial condition $ (\chi^u_0)_u \in L^2_{m}(\mathbb{R}^n) $ and for every $u \in U$, $\chi^u_0$ is $\mathcal{F}^u_0$-measurable. Finally, $x \mapsto b^u(t, x, y, z, r, \xi)$ and $y \mapsto f^u(t, x, y, z, r,\xi)$ are continuous for every $u\in U$. \label{ass:2.1.1}
    
    \item  \label{ass:2.1.2}There exist constants \( \lambda_1, \lambda_2  \in \mathbb{R} \) such that, for every \( u \in U \) and \( t \in [0,T] \),
\[
\langle b^u( t, x_1, y, z, r, \xi) - b^u( t, x_2, y, z, r,\xi),x_1-x_2 \rangle \leq \lambda_1  |x_1-x_2|^2,
\]
\[
\langle f^u( t, x, y_1, z,r, \xi) - f^u( t, x, y_2, z, r,\xi), y_1-y_2 \rangle \leq \lambda_2  |y_1-y_2|^2.
\] 

    \item \label{ASS: bound}There exist positive constants $\rho$, such that, for every $u \in U$ and \( t \in [0,T] \),
\[
|b^u(t, x, y, z,r, \xi)| \le |b^u(t, 0, y, z, r, \xi)| + \rho (1 + |x|),
\]\[
|f^u(t, x, y, z,r, \xi)| \le |f^u(t, x, 0, z, r, \xi)| + \rho (1 + |y|).
\]

    \item \label{Ass: coef}There exist positive constants $\rho_i, \mu_i, i = 1,2, 3$, such that, for every $u \in U$ and \( t \in [0,T] \),
\begin{align*}
&|b^u(t, x, y_1, z_1, r, \eta_1) - b^u(t, x, y_2, z_2, r, \eta_2)| \le \rho_1 |y_1 - y_2| + \rho_2 |z_1 - z_2| 
+ \rho_{3} W_{2,m}(\eta_1, \eta_2),\\
&|f^u(t, x_1, y, z_1, r, \eta_1) - f^u(t, x_2, y, z_2, r, \eta_2)| \le \mu_1 |x_1 - x_2|  + \mu_2 |z_1 - z_2| 
+ \mu_{3} W_{2,m}(\eta_1, \eta_2).
\end{align*} 
\item \label{Ass: coefsig} There exist positive constants $w_j, j = 1, 2, 3, 4$, such that, for every $u \in U$ and \( t \in [0,T] \),
    \[
    \begin{aligned}
    &|\sigma^u(t, x_1, y_1, z_1, r, \eta_1) - \sigma^u(t, x_2, y_2, z_2, r,\eta_2)|^2  \leq w_1^2 |x_1-x_2|^2 +w_2^2 |y_1-y_2|^2+\\& \quad w_3^2|z_1-z_2|^2 + w_4^2 W^2_{2,m}(\eta_1, \eta_2).
    \end{aligned}
    \]
    \item \label{Ass: terminal}There exist positive constants $\rho_4, \rho_5 $ such that, for every $u \in U$,
$$
    |G^u(x_1, \zeta_1) - G^u(x_2, \zeta_2)|^2 \le  \rho_4^2 |x_1-x_2|^2 + \rho_5^2 W^2_{2,m}(\zeta_1, \zeta_2).
$$
    \item \label{ass:2.1.6} Let 
\(
\boldsymbol{v}_{r, 
\mu} = (0, 0, 0, r, (\mu^u)_u),
\)
for an arbitrary $(\mu^u)_u \in \mathcal{P}_m^2(\mathbb{R}^n \times \mathbb{R}^l \times \mathbb{R}^{l \times d} \times \mathbb{R}^k) $, where
\(
\mu^u \circ \pi^{-1}_{1,2,3}
\)
is the Dirac measure at 0 on \(\mathbb{R}^n \times \mathbb{R}^l \times \mathbb{R}^{l \times d}\) for each $u \in U$.
Let $\delta^{u}_0$ be the Dirac measure at $0$ on $
\mathbb{R}^n$ for each $u \in U$, we assume
\[
   (b^u, \sigma^u, f^u)_u(\cdot,  \boldsymbol{v}_{r, 
\mu})  \in \mathcal{M}^2_{m}(\mathbb{R}^n \times \mathbb{R}^{n \times d} \times \mathbb{R}^l) \quad \mathrm{and} \quad  \bigl(G^u(0, (\delta^{\tilde{u}}_0)_{\tilde{u}})\bigr)_u\in L^2_{m}( \mathbb{R}^l). 
\] 
\end{enumerate}
\end{assumption}
 
Similarly to \cite{Pardoux}, smallness conditions also lead to well-posedness of \eqref{eq:3.2}.
\begin{theorem} \label{Theorem 4.2}
Let Assumption \ref{Assumption 4.1} hold. We denote $\bar{\lambda}_1 = \lambda - 2\lambda_1 - C_1^{-1} \rho_1 - C_2^{-1} \rho_2 -(2+ C_3^{-1}  + C_4^{-1})\rho_{3} - w_1^2 - w_{4}^2$ and $ \bar{\lambda}_2 = -\lambda - 2\lambda_2 - K_1^{-1} \mu_1 - K_2^{-1}(\mu_2 + \mu_{3}) - (2+K_3^{-1}) \mu_{3}$, where $\lambda \in \mathbb{R}$,  $C_i >0, i = 1,2,3,4$ and $K_j >0 , j = 1,2,3$ are constants to be specified. If
\[
2(\lambda_1 + \lambda_2) < -2\rho_{3} - w_1^2 - w_{4}^2 - (\mu_2 + \mu_{3})^2 - 2\mu_{3},
\]
then we can choose appropriate $\lambda$, $K_2$, and sufficiently large $C_1$, $C_2$, $C_3$, $C_4$, $K_1$, $K_3$ so that
\[
1 - K_2 \mu_2 - K_2 \mu_{3} > 0 \quad \text{and} \quad \bar{\lambda}_1, \bar{\lambda}_2 > 0.
\]
Provided this holds, the constant $\theta = 1/ [( \frac{1}{\bar{\lambda}_2} + \frac{1}{1 - K_2 \mu_2 - K_2 \mu_{3}})
( \rho^2_4 + \rho_5^2 + \frac{K_1 \mu_1 + K_3 \mu_{3}}{\bar{\lambda}_1})]$, which is positive and independent of $T$, is such that if
\[
C_1\rho_1+w_2^2+ C_3\rho_{3}+w_{4}^2 \vee C_2\rho_2+ w_3^2+ C_4\rho_{3}+w_{4}^2 \in [0, \theta),
\] 
then the FBSDE system \eqref{eq:3.2} has a unique strong solution \( (X^u, Y^u, Z^u,\Lambda^u)_u \in \mathcal{S}^2_m\times \mathcal{H}^2_m \). 
\end{theorem}

The proof of Theorem \ref{Theorem 4.2} is the subject of Section \ref{sect:well-posedness}.
\begin{remark}\label{Remark 3.1}
   The above conditions are analogous to those in \cite{Pardoux}, where also local well-posedness result is derived $\theta$ depending on $T$, with $T$ sufficiently small. In that case slightly weaker assumptions are needed (see, e.g., \cite[Theorem 3.1]{Pardoux}). One can have the analogous statement for our system, but we leave this out for simplicity of the presentation.
\end{remark}

If the system is only coupled through the marginals $(\mathrm{Law}(X^u_t))_u$, $(\mathrm{Law}(Y^u_t))_u$, $(\mathrm{Law}(Z^u_t))_u$, then following assumptions can be relevant for the smallness conditions.
\renewcommand{\theassumption}{2.1'}
\begin{assumption} \label{ass:alternative}
Suppose there are constants $\rho_{3i}$, $\mu_{3i}$ and $w_{4i}$, for $i = 1,2,3$, such that, for every $u \in U$ and \( t \in [0,T] \), 
\begin{align*}\label{ass:linear} 
&\bigl|b^u(t,x,y_1,z_1,r,\eta_1)-b^u(t,x,y_2,z_2,r,\eta_2)\bigr|
 \le \rho_1|y_1-y_2|+\rho_2|z_1-z_2|
   + \rho_{31} W_{2,m}(\eta_1^{(1)},\eta_2^{(1)}) \\& \quad +\rho_{32} W_{2,m}(\eta_1^{(2)},\eta_2^{(2)})+\rho_{33} W_{2,m}(\eta_1^{(3)},\eta_2^{(3)}),\\
&\bigl|f^u(t,x_1,y,z_1,r,\eta_1)-f^u(t,x_2,y,z_2,r,\eta_2)\bigr|
 \le \mu_1|x_1-x_2|+\mu_2|z_1-z_2|
   + \mu_{31} W_{2,m}(\eta_1^{(1)},\eta_2^{(1)}) \\& \quad + \mu_{32} W_{2,m}(\eta_1^{(2)},\eta_2^{(2)}) +\mu_{33} W_{2,m}(\eta_1^{(3)},\eta_2^{(3)}),\\
    &|\sigma^u(t, x_1, y_1, z_1, r, \eta_1) - \sigma^u(t, x_2, y_2, z_2, r,\eta_2)|^2  \leq w_1^2 |x_1-x_2|^2 + w_2^2 |y_1-y_2|^2 +  w_3^2 |z_1-z_2|^2   \\& \quad +  w^2_{41} W^2_{2,m}(\eta_1^{(1)},\eta_2^{(1)}) +w^2_{42} W^2_{2,m}(\eta_1^{(2)},\eta_2^{(2)})+w^2_{43} W^2_{2,m}(\eta_1^{(3)},\eta_2^{(3)}). 
\end{align*}
\end{assumption}

\renewcommand{\theassumption}{\thesection.\arabic{assumption}}
Proceeding analogously to the proof of Theorem \ref{Theorem 4.2} and adapting the arguments in the proof of \cite[Theorem 3.1]{CHEN2023105550}, we obtain the following variant of Theorem \ref{Theorem 4.2}.

\begin{theorem} \label{Theorem 4.4}
Suppose that instead of \ref{Ass: coef} and \ref{Ass: coefsig} of Assumption \ref{Assumption 4.1}, we have Assumption \ref{ass:alternative} and the rest of Assumption \ref{Assumption 4.1} holds. We denote $\bar{\lambda}_1 = \lambda - 2\lambda_1 - C_1^{-1} \rho_1 - C_2^{-1} \rho_2 - C_3^{-1} \rho_{32} - C_4^{-1} \rho_{33} - 2\rho_{31} - w_1^2 - w_{41}^2$ and $ \bar{\lambda}_2 = -\lambda - 2\lambda_2 - K_1^{-1} \mu_1 - K_2^{-1}(\mu_2 + \mu_{33}) - K_3^{-1} \mu_{31} - 2\mu_{32}$, where $\lambda \in \mathbb{R}$,  $C_i >0, i = 1,2,3,4$ and $K_j >0 , j = 1,2,3$ are constants to be specified. If
\[
2(\lambda_1 + \lambda_2) < -2\rho_{31} -w_1^2 - w_{41}^2 - (\mu_2 + \mu_{33})^2 - 2\mu_{32},
\]
then we can choose appropriate $\lambda$, $K_2$, and sufficiently large $C_1$, $C_2$, $C_3$, $C_4$, $K_1$, $K_3$ such that
\[
1 - K_2 \mu_2 - K_2 \mu_{33} > 0 \quad \text{and} \quad \bar{\lambda}_1, \bar{\lambda}_2 > 0.
\]
Provided this holds, the constant $\theta = 1/ [( \frac{1}{\bar{\lambda_2}} + \frac{1}{1 - K_2 \mu_2 - K_2 \mu_{33}})
(  \rho^2_4 + \rho_5^2 + \frac{K_1 \mu_1 + K_3 \mu_{31}}{\bar{\lambda_1}})],$ which is positive and independent of $T$, is such that if 
\[
C_1\rho_1+w_2^2+ C_3\rho_{32}+w_{42}^2 \vee C_2\rho_2+ w_3^2+ C_4\rho_{33}+w_{43}^2 \in [0, \theta),
\] 
then the FBSDE system \eqref{eq:3.2} admits a unique strong solution \( (X^u, Y^u, Z^u,\Lambda^u)_u \in \mathcal{S}^2_m\times \mathcal{H}^2_m \). 
\end{theorem}

\section{The heterogeneous mean-field control problem} \label{section 4}
In this section, we establish a Pontryagin stochastic maximum principle and derive a corresponding verification theorem for the controlled system of FBSDEs
\begin{equation}
    \begin{cases}
&\mathrm{d}X_t^u = b^u\bigl(t,X_t^u, Y_t^u, Z^u_t, \alpha^u_t, (\mathbb{P}^{\tilde{u}}_{t})_{\tilde{u}}\bigr)\mathrm{d}t+ \sigma^u\bigl(t,X_t^u, Y_t^u, Z^u_t, \alpha^u_t, (\mathbb{P}^{\tilde{u}}_{t})_{\tilde{u}}\bigr)\mathrm{d}B^u_t \\
&-\mathrm{d}Y^u_t = f^u\bigl(t,X_t^u, Y_t^u, Z^u_t, \alpha^u_t, (\mathbb{P}^{\tilde{u}}_{t})_{\tilde{u}}\bigr)\mathrm{d}t -Z^u_t \mathrm{d}B^u_t\\
\end{cases}  \label{eq:4.1}
\end{equation}
with initial conditions $X_0^u = \chi_0^u$ and terminal conditions
$Y^u_T = G^u(X^u_T,(\mathbb{P}^{\tilde{u}}_{T,1})_{\tilde{u}})$, for $u \in U$.
Writing $\boldsymbol{\alpha} = (\alpha^u)_u$, we recall that the global cost function is
\[
J(\boldsymbol{\alpha}) = \int\mathbb{E}  \Bigl[  \int_0^T \ell^u\bigl(t,X_t^u, Y_t^u, Z^u_t, \alpha^u_t, (\mathbb{P}^{\tilde{u}}_{t})_{\tilde{u}}\bigr)\mathrm{d}t 
+ h^u\bigl(X^u_T,(\mathbb{P}^{\tilde{u}}_{T,1})_{\tilde{u}}\bigr) + g^u(Y^u_0)   \Bigr]\mathrm{d}m(u).
\]
Let \( \mathcal{A} \) be a non-empty convex subset of \( \mathbb{R}^k \). Then, the set of admissible controls is
\[
\mathcal{A}_{\text{ad}} :=  \{ \boldsymbol{\alpha} |\boldsymbol{\alpha} \in \mathcal{M}^2_{m}( \mathbb{R}^k) , \alpha^u_t= \alpha^u(t,\chi_0^u,B^u_{\cdot \wedge t}), \alpha^u_t \in \mathcal{A},  \text{ for $u\in U$}, \ t \in [0,T],\ \mathbb{P}\text{-a.s.}  \},
\] 
where $(u,t,x_0,w) \mapsto \alpha^u(t, x_0, w_{\cdot\land t})$ must be Borel measurable on $U \times [0,T] \times \mathbb{R}^n \times C([0,T], \mathbb{R}^d).$

\subsection{Differentiability along the Markov kernels} \label{Lion construct}

Consider a measurable function of the form
\begin{equation}\label{eq:b_general_form}
(u,t,x,y,z,a,(P^{\tilde{u}})_{\tilde{u}})\mapsto \gamma^u(t,x,y,z,a,(P^{\tilde{u}})_{\tilde{u}}) \in \mathbb{R}^r,
\end{equation}
for $(u,t,x,y,z,a)\in U \times [0,T] \times  E$ and $(P^{\tilde{u}})_{\tilde{u}} \in \mathcal{P}^2_m(E)$, where $E=\mathbb{R}^n \times \mathbb{R}^l \times \mathbb{R}^{l \times d} \times \mathbb{R}^k$ and $r$ is a positive integer. We need a notion of differentiability in the direction of the family of probability measures $(P^{u})_{u}$ or a family of its marginals (e.g., the first marginal of each $P^{u}$ on $\mathbb{R}^n$). It turns out that we can arrive at this based on the standard machinery of $L$-differentiability. 

Let $(\Omega',\mathcal{{F}}',\mathbb{{P}}')$
support a random variable $\mathfrak{u}:\Omega^\prime \rightarrow U $ with $\text{Law}(\mathfrak{u})=m$
and an independent uniform random variable $I:\Omega^\prime \rightarrow (0,1)$. Given $\mathbb{Q}(O\times D)= \int_D P^u(O)\mathrm{d}m(u)$, we can find a measurable function $\tilde{h}:U\times(0,1) \rightarrow  E$
such that $(X',Y',Z',A'):=\tilde{h}(\mathfrak{u},I)$
satisfies $
((X',Y',Z',A'),\mathfrak{u})\sim \mathbb{Q}$ (see e.g.~\cite[Lemma 4.22]{kallenberg}).
Conversely, any random variable $(X',Y',Z',A'): \Omega'\rightarrow E$,
induces a joint law $\mathbb{Q}=\mathrm{Law}((X',Y',Z',A'),\mathfrak{u})$ whose last marginal is $m$. For each $u\in U$, it is therefore natural to `lift' \eqref{eq:b_general_form} to the mapping
\begin{equation}\label{eq:lift}
\tilde{\gamma}^u(t,x,y,z,a,(X',Y',Z',A')):=\gamma^u(t,x,y,z,a,({Q}^{\tilde{u}})_{\tilde{u}})
\end{equation}
on $L^{2}((\Omega',\mathcal{{F}}',{P}'); E)$, where the Markov kernel $({Q}^{u})_{u}$ is the conditional law of $(X',Y',Z',A')$ given $\mathfrak{u}$ (which exists, e.g., by~\cite[Theorem 8.5]{kallenberg}). Now suppose we have $L$-differentiability in the usual sense that
\[
(X',Y',Z',A')\mapsto\tilde{\gamma}^u(t,x,y,z,a,(X',Y',Z',A'))
\]
 is Fr\'echet differentiable on $L^{2}((\Omega',\mathcal{{F}}',\mathbb{{P}}'); E)$. Let $\Theta^{i}:=(t,x,y,z,a,(P^{i,u})_{u})$ be given, for $i=1,2$, and let $\Xi^{i}:=(X^{\prime,i},Y^{\prime,i},Z^{\prime,i},A^{\prime,i})$ correspond to $(P^{i,u})_{u}$ in the lift \eqref{eq:lift}. Then, the assumed $L$-differentiability (see e.g.~\cite[Proposition 5.24]{CarmonaDelarue2018}) gives that
\begin{equation}\label{eq:L-diff}
\tilde{\gamma}^u(t,x,y,z,a,\Xi^{2})-\tilde{\gamma}^u(t,x,y,z,a,\Xi^{1})
=\mathbb{{E}}  [\partial_{P}\gamma^u(\Theta^{1})(\Xi^{1}) \cdot (\Xi^2 - \Xi^1)  ]+o(||\Xi^2-\Xi^1||_2)
\end{equation}
as $||\Xi^2-\Xi^1||_2 \rightarrow 0$,
for some functional
\begin{equation}
\partial_{P}\gamma^u(\Theta^{1})(\cdot):\mathbb{{R}}^{n+l+l\cdot d+k} \rightarrow \mathbb{R}^{r \times (n+l+l\cdot d+k)}.
\end{equation}
This functional is called the $L$-derivative, for the lift \eqref{eq:lift}, and $\partial_{P}\gamma^u(\Theta^{1})(\Xi^{1})$ is in $L^2$. By the structure of the lift, we have
\begin{align*}
    \mathbb{{E}}[\partial_{P}\gamma^u(\Theta^{1})(\Xi^{1})\cdot (\Xi^{2}- \Xi^{1}) ] &=\mathbb{{E}}^{\mathfrak{u}}\bigl[\mathbb{E}[\partial_{P}\gamma^u(\Theta^{1})(\Xi^{1})(\Xi^{2}- \Xi^{1})\mid \mathfrak{u}] \bigr] \\& =  \int \mathbb{E}  [ \partial_{P}\gamma^u(\Theta^{1})(\Xi^{1,\tilde{u}}) (\Xi^{2,\tilde{u}}- \Xi^{1,\tilde{u}})    ]\mathrm{d} m (\tilde{u}),
\end{align*}
for random variables
\begin{equation}\label{eq:xi^u}
\Xi^{i,u} = (X^{\prime,i,u},Y^{\prime,i,u},Z^{\prime,i,u},A^{\prime,i,u}) \quad \text{with} \quad \text{Law}(\Xi^{i,u})=P^{i,u},\quad u\in U,
\end{equation}
for $i=1,2$, which we may view as providing a family of $L^2$-lifts for $(P^{i,u})_{u}$ in accordance with \eqref{eq:lift}.
The same applies to
 $||\Xi^1-\Xi^2||_2$. Therefore,
\eqref{eq:L-diff} can be expressed as
\begin{align}\label{eq:generalised_L-diff}
\gamma^u(\Theta^2)-\gamma^u(\Theta^1)
&=\int \mathbb{{E}}  [\partial_{P}\gamma^u(\Theta^{1})(\Xi^{1,\tilde{u}}) (\Xi^{2,\tilde{u}}-\Xi^{1,\tilde{u}})  ]\mathrm{d}m(\tilde{u}) \nonumber \\
&\qquad \qquad +o  \bigl(\int||\Xi^{1,\tilde{u}}-\Xi^{2,\tilde{u}}||_2 \mathrm{d}m(\tilde{u})  \bigr),
\end{align}
as $\int||\Xi^{1,\tilde{u}}-\Xi^{2,\tilde{u}}||_2 \mathrm{d}m(\tilde{u})  \rightarrow 0$, for $(\Xi^{i,u})_u$ given by \eqref{eq:xi^u}, with
\begin{equation}\label{eq:diff_L2}
\int \mathbb{{E}}  [|\partial_{P}\gamma^u(\Theta^{1})(\Xi^{1,\tilde{u}})|^2]\mathrm{d}m(\tilde{u}) <\infty.
\end{equation}

\begin{definition}[$L_m$-differentiability]\label{def:Lm-derivative} For any given $t\in [0,T]$, $(x,y,z,a)\in E$, and $(P^{\tilde{u}})_{\tilde{u}}\in \mathcal{P}^2_m(E)$, we write $\Theta = (t,x,y,z,a, (P^{\tilde{u}})_{\tilde{u}})$.  We will say that a measurable function \[
(u,\Theta)\mapsto \gamma^u(\Theta)
\]is $L_m$-differentiable if \eqref{eq:generalised_L-diff}--\eqref{eq:diff_L2} holds for all points $(t,x,y,z,a)\in[0,T]\times E$, all indices $u\in U$ in the support of $m$, and any pair of Markov kernels $(P^{1,u})_{u}$ and $(P^{2,u})_{u}$ in $\mathcal{P}_m^2(E)$. In that case, we refer to the family $(\partial_{P}\gamma^u(\Theta^1))_{u}$ of mappings
\begin{equation} \label{eq:Lm_derivative}
    \partial_{P}\gamma^u(\Theta^{1})(\cdot):\mathbb{{R}}^{n+l+l\cdot d+k} \rightarrow \mathbb{R}^{r \times (n+l+l\cdot d+k)}
\end{equation}
as the $L_m$-derivative of $(u,\Theta)\mapsto \gamma^u(\Theta)$ at $\Theta^1=(t,x,y,z,a,(P^{1,\tilde{u}})_{\tilde{u}})$.
\end{definition}

The key property that we exploit in our arguments is \eqref{eq:generalised_L-diff}. One can of course also consider partial $L_m$-derivatives, for $\Theta^1$ and $\Theta^2$ that only differ in the direction of a given marginal. When this is relevant, we shall denote by $\partial_{P^{(1)}}\gamma^u$, $\partial_{P^{(2)}}\gamma^u$, $\partial_{P^{(3)}}\gamma^u$, and $\partial_{P^{(4)}}\gamma^u$, the partial $L_m$-derivatives with respect to the marginals represented by $X^\prime$, $Y^\prime$, $Z^\prime$, and $A^\prime$ in the above.

\subsection{The maximum principle}\label{sect:max_principle}

We make the following assumption to prove the maximum principle.

\begin{assumption}\label{Ass:5.2}
We fix the notation \( x, x_1, x_2 \in \mathbb{R}^n \), \( y, y_1, y_2 \in \mathbb{R}^l \), \( z, z_1, z_2 \in \mathbb{R}^{l \times d} \), \( r, r_1, r_2 \in \mathbb{R}^k \), \(\zeta, \zeta_1, \zeta_2 \in \mathcal{P}^2_m(\mathbb{R}^n) \),  and \( \xi, \eta_1, \eta_2 \in \mathcal{P}^2_m(E) \), where $E=\mathbb{R}^n \times \mathbb{R}^l \times \mathbb{R}^{l \times d} \times \mathbb{R}^k$. Let
$\Delta \varphi = \varphi_1 - \varphi_2$, for $\varphi = x, y, z, r$, $
\theta= (x, y, z, r, \xi)$, $\theta^{x}_i = (x, y_i, z_i, r_i, \eta_i)$, $ \theta^{y}_i := (x_i, y, z_i, r_i, \eta_i)$, $\theta_i = (x_i, y_i, z_i, r_i, \eta_i)$, $ i = 1,2$.
Assume the following conditions hold:

\begin{enumerate}
     \item The functions $b, \sigma, f, l: U \times [0,T] \times E \times \mathcal{P}^2_m(E) \rightarrow \mathbb{R}^n , \mathbb{R}^{n \times d}, \mathbb{R}^l, \mathbb{R}$ are Borel measurable. The function $h: U \times  \mathbb{R}^n \times \mathcal{P}^2_m(\mathbb{R}^n) \rightarrow \mathbb{R}$ is Borel measurable and the function $g: U \times \mathbb{R}^l  \rightarrow \mathbb{R}$ is Borel measurable. Moreover,  we assume the initial condition $ (\chi^u_0)_u \in L^2_{m}(\mathbb{R}^n) $ and for every $u \in U$, $\chi^u_0$ is $\mathcal{F}^u_0$-measurable. Finally, $x \mapsto b^u(t, x, y, z, r, \xi)$ and $y \mapsto f^u(t, x, y, z, r,\xi)$ are continuous for every $u\in U$.
    \item For every $u\in U$,  the functions \( b^u, \sigma^u, f^u, \ell^u \) are continuously differentiable with respect to \( (x, y, z, r) \),  the functions \( G^u, h^u \) are continuously  differentiable with respect to \( x \), and \( g^u \) is continuously  differentiable with respect to \( y\). We also require that the derivatives \( b^u_x \) and \( g^u_y \) are uniformly bounded. Moreover, the functions \( b, \sigma, f, \ell \) are continuously $L_m$-differentiable with respect to \( \xi \) in the sense that \eqref{eq:Lm_derivative} exist and are continuous for $\gamma = b, \sigma, f, \ell$. Finally, we require the functions \( G, h \) are continuously $L_m$-differentiable with respect to \( \zeta \). 
    
    \item  Let 
\(
\boldsymbol{v}_{r, 
\mu} := \bigl(0, 0, 0, r, (\mu^u)_u \bigr),
\)
for an arbitrary $(\mu^u)_u \in \mathcal{P}_m^2(E) $ such that \(
\mu^u \circ \pi^{-1}_{1,2,3}
\)
is the Dirac measure at 0 on \(\mathbb{R}^n \times \mathbb{R}^l \times \mathbb{R}^{l \times d}\) for all $u \in U$.
Let $\delta^{s,u}_0$ be the Dirac measure at $0$ on $
\mathbb{R}^s$ for each $u \in U$. Then we require
    \[
   (b^u, \sigma^u, f^u)_u(\cdot,  \boldsymbol{v}_{r, 
\mu})  \in \mathcal{M}^2_{m}(\mathbb{R}^n \times \mathbb{R}^{n \times d} \times \mathbb{R}^l)\quad \mathrm{and} \quad  \bigl(G^u(0, (\delta^{n,u}_0)_u)\bigr)_u\in L^2_{m}( \mathbb{R}^l). 
    \]

    \item There exists a constant \( L > 0 \) such that, for every $u\in U$ and $t \in [0,T]$,
\begin{align*}
     |\ell^u(t,x, y, z, r, \xi)|   &\leq L \bigl(1 + |x|^2 +|y|^2+|z|^2+|r|^2+ W^2_{2,m}(\xi, (\delta^{n\times l \times(l\times d)\times k,u}_0)_u)\bigr),\\
     |h^u(x,\zeta)|  &\leq L \bigl(1 + |x|^2 +W^2_{2,m}(\zeta, (\delta^{n,u}_0)_u) \bigr),\\
    |g^u(y)|& \leq L(1+|y|^2).
 \end{align*}
    
    \item Denote $M^{\top}$ as the transpose of $M$ and $I$ as identity matrix, for every $u\in U$ and $t \in [0,T]$,
\begin{align*}
&\frac{1}{2}  ( b^u_x + (b^{u}_x)^{\top}  )(t, \theta) \leq \lambda_1 I_n, \quad
\frac{1}{2}  ( f^u_y + (f^{u}_y)^{\top} )(t, \theta) \leq \lambda_2 I_l,\\
&|b^u(t, \theta^{x}_1) - b^u(t, \theta^{x}_2)|
\leq \rho_1 |\Delta y| + \rho_2 |\Delta z| + \rho_6 |\Delta r| 
+ \rho_{3} W_{2,m}(\eta_1, \eta_2),  \\
&|f^u(t, \theta^{y}_1) - f^u(t, \theta^{y}_2)|
\leq \mu_1 |\Delta x| + \mu_2 |\Delta z| + \mu_4 |\Delta r| 
+ \mu_{3} W_{2,m}(\eta_1, \eta_2) , \\
&|\sigma^u(t, \theta_1) - \sigma^u(t, \theta_2)|^2
\leq w_1^2|\Delta x|^2 + w^2_2 |\Delta y|^2 + w^2_3 |\Delta z|^2 + w^2_5 |\Delta r|^2  + w_{4}^2 W^2_{2,m}(\eta_1, \eta_2),  \\
&|G^u(x_1, \zeta_1) - G^u(x_2, \zeta_2)|^2
\leq \rho_4^2 |\Delta x|^2 + \rho^2_5 W^2_{2,m}(\zeta_1, \zeta_2).
\end{align*}
where $\lambda_1, \lambda_2 \in \mathbb{R}$, $ \rho_i$, for $i = 1,\cdots, 6$, $\mu_j$, for $ j = 1, \cdots, 4$, and $w_k$, for $ k = 1, \cdots, 5$ are positive constants.
\item Choose $0<K_2 <\frac{1}{\mu_2+\mu_{3}}$. For notational simplicity (in anticipation of the variational equations), set $\hat{w}_i := \sqrt{6}w_i$, for $i = 1,2,3,4$ and $\hat{\rho}_j := \sqrt{2}\rho_j$, for $j =4,5$. We assume
\begin{align*}
    &2\lambda_1+    2 \lambda_2  <- 2\rho_{3}-\hat{w}_1^2-\hat{w}_{4}^2-2\mu_{3}-(\mu_2+\mu_{3})^2,
\end{align*}
and choose $\lambda,C_1,C_2,C_3,C_4,K_1,K_3$  such that 
$\bar{\lambda}_{11} = \lambda-2\lambda_1 - C_1^{-1}\rho_1 - C_2^{-1} \rho_2 - ( 2+C_3^{-1}  + C_4^{-1})\rho_{3}-\hat{w}_1^2 - \hat{w}_{4}^2>0$ and $\bar{\lambda}_{12} =-\lambda - 2\lambda_2 - K_1^{-1} \mu_1 - K_2^{-1}(\mu_2 + \mu_{3}) - (2+K_3^{-1}) \mu_{3} >0 $.
Let  $\theta_{1} = 1/\bigl[( \frac{1}{\bar{\lambda}_{12}} + \frac{1}{1 - K_2 \mu_2 - K_2 \mu_{3}} )
 ( \hat{\rho}^2_4 + \hat{\rho}_5^2 + \frac{K_1 \mu_1 + K_3 \mu_{3}}{\bar{\lambda}_{11}} )\bigr],$ and we assume \[
C_1\rho_1+\hat{w}_2^2+ C_3\rho_{3}+\hat{w}_{4}^2  \vee C_2\rho_2+ \hat{w}_3^2+ C_4\rho_{3}+\hat{w}_{4}^2  \in [0, \theta_{1}).
\] 
\end{enumerate}
\end{assumption}
\begin{remark} \label{Remark:5.1}
   By the same arguments as in \cite[Remark 4.1]{CHEN2023105550}, for $u \in U$, if $b^u_x$ is continuous, we have that \[
\frac{1}{2}  ( b^u_x + (b^u_x)^\top  )(t, \theta) \leq \lambda_1 I_n
\] is equivalent to \[
\langle b^u( t, x_1, y, z, r, \xi) - b^u( t, x_2, y, z, r, \xi), x_1-x_2 \rangle \leq \lambda_1  |x_1-x_2|^2.
\] 
The analogous observation also holds for the coefficient $f^u$, for $u \in U$.
\end{remark} 

Under the assumption on the control $\boldsymbol{\alpha}$, we may replace
$\Lambda^u_t$ in system \eqref{eq:3.2} by $\alpha^u_t$ for all $u \in U$.  Applying Theorem \ref{Theorem 4.2} and Remark \ref{Remark:5.1} then yields the following result.

\begin{theorem} \label{Theorem for FBSDE}
    Let Assumption \ref{Ass:5.2} hold. Then there exists a unique strong solution for the controlled FBSDE \eqref{eq:4.1} under any admissible control $\boldsymbol{\alpha}$.
\end{theorem}

To derive the necessary condition for the stochastic maximum principle, we let \( \boldsymbol{\hat{\alpha}} := (\hat{\alpha}^u)_u \) be an optimal control with corresponding optimal trajectory \( (X, Y, Z) \). Let \( \boldsymbol{\pi} \) be such that \( \boldsymbol{\hat{\alpha}} + \boldsymbol{\pi} \in \mathcal{A}_{\mathrm{ad}} \). Since $\mathcal{A}$ is convex, we have $\hat{\alpha}^u + \varepsilon \pi^u \in \mathcal{A}$ for any $\varepsilon \in [0,1]$ and all $u \in U$. Consequently, $\boldsymbol{\hat{\alpha}} + \varepsilon \boldsymbol{\pi} \in \mathcal{A}_{\mathrm{ad}}$ for any $\varepsilon \in [0,1]$. To simplify the notation, we write
\[
\Theta^u_t := \bigl(X^u_t, Y^u_t, Z^u_t, \hat{\alpha}^u_t, (\mathbb{P}^{\tilde{u}}_{t})_{\tilde{u}}\bigr)\quad \text{and} \quad \vartheta^u_t := (X^u_t, Y^u_t, Z^u_t, \hat{\alpha}^u_t).
\]
We then introduce the following system of variational equations
\begin{equation}
 \label{eq::variation}
\begin{cases}
\mathrm{d}X_t^{\prime,u} &= \Big[
    b^u_x(t, \Theta^u_t) X_t^{\prime,u} + b^u_y(t, \Theta^u_t) Y_t^{\prime,u} + b^u_z(t, \Theta^u_t) Z_t^{\prime,u} + b^u_{\alpha}(t, \Theta^u_t) \pi^u_t \\
    &\quad + \int \tilde{\mathbb{E}}[ \partial_{P^{(1)}}b^u(t, \Theta^u_t)\bigl(\tilde{\vartheta}^{\tilde{u}}_t\bigr) \tilde{X}_t^{\prime,\tilde{u}} + \partial_{P^{(2)}}b^u(t, \Theta^u_t)\bigl(\tilde{\vartheta}^{\tilde{u}}_t\bigr)\tilde{Y}_t^{\prime,\tilde{u}} 
    \\
    &\quad +  \partial_{P^{(3)}}b^u(t, \Theta^u_t)\bigl(\tilde{\vartheta}^{\tilde{u}}_t\bigr) \tilde{Z}_t^{\prime,\tilde{u}}+  \partial_{P^{(4)}}b^u(t, \Theta^u_t)\bigl(\tilde{\vartheta}^{\tilde{u}}_t\bigr) \tilde{\pi}_t^{\tilde{u}} ] \mathrm{d}m(\tilde{u})
\Big] \mathrm{d}t \\
&\quad + \Big[
    \sigma^u_x(t, \Theta^u_t) X_t^{\prime,u} + \sigma^u_y(t, \Theta^u_t) Y_t^{\prime,u} + \sigma^u_z(t, \Theta^u_t) Z_t^{\prime,u} + \sigma^u_{\alpha}(t, \Theta^u_t) \pi^u_t \\
    &\quad + \int \tilde{\mathbb{E}}[ \partial_{P^{(1)}}\sigma^u(t, \Theta^u_t)\bigl(\tilde{\vartheta}^{\tilde{u}}_t\bigr) \tilde{X}_t^{\prime,\tilde{u}} + \partial_{P^{(2)}}\sigma^u(t, \Theta^u_t)\bigl(\tilde{\vartheta}^{\tilde{u}}_t\bigr) \tilde{Y}_t^{\prime,\tilde{u}} 
    \\
    &\quad + \partial_{P^{(3)}}\sigma^u(t, \Theta^u_t)\bigl(\tilde{\vartheta}^{\tilde{u}}_t\bigr) \tilde{Z}_t^{\prime,\tilde{u}} 
    +\partial_{P^{(4)}}\sigma^u(t, \Theta^u_t)\bigl(\tilde{\vartheta}^{\tilde{u}}_t\bigr) \tilde{\pi}_t^{\tilde{u}} ] \mathrm{d}m(\tilde{u})
\Big] \mathrm{d}B^u_t, \\
\mathrm{d}Y_t^{\prime,u} &= -\Big[
    f^u_x(t, \Theta^u_t) X_t^{\prime,u} + f^u_y(t, \Theta^u_t) Y_t^{\prime,u} + f^u_z(t, \Theta^u_t) Z_t^{\prime,u} + f^u_{\alpha}(t, \Theta^u_t) \pi^u_t \\
    &\quad + \int \tilde{\mathbb{E}}[ \partial_{P^{(1)}}f^u(t, \Theta^u_t)\bigl(\tilde{\vartheta}^{\tilde{u}}_t\bigr)\tilde{X}_t^{\prime,\tilde{u}} + \partial_{P^{(2)}}f^u(t, \Theta^u_t)\bigl(\tilde{\vartheta}^{\tilde{u}}_t\bigr) \tilde{Y}_t^{\prime,\tilde{u}} 
    \\
    &\quad + \partial_{P^{(3)}}f^u(t, \Theta^u_t)\bigl(\tilde{\vartheta}^{\tilde{u}}_t\bigr) \tilde{Z}_t^{\prime,\tilde{u}} 
    + \partial_{P^{(4)}}f^u(t, \Theta^u_t)\bigl(\tilde{\vartheta}^{\tilde{u}}_t\bigr) \tilde{\pi}_t^{\tilde{u}} ] \mathrm{d}m(\tilde{u})
\Big]\mathrm{d}t + Z_t^{\prime,u} \mathrm{d}B^u_t,
\end{cases}
\end{equation}
with initial conditions $X_0^{\prime,u} = 0$ and terminal conditions
\[
Y_T^{\prime,u} = G^u_x\bigl(X^u_T, (\mathbb{P}^{\tilde{u}}_{T,1})_{\tilde{u}}\bigr) X_T^{\prime,u} + \int
\tilde{\mathbb{E}}\Big[ \partial_{P^{(1)}}G^u\bigl(X^u_T,(\mathbb{P}^{\tilde{u}}_{T,1})_{\bar{u}}\bigr)\bigl(\tilde{X}^{\tilde{u}}_T\bigr) \tilde{X}_T^{\prime,{\tilde{u}}} \Big] \mathrm{d}m({\tilde{u}}),
\]
where
$$
\tilde{\vartheta}^{u}_t := (\tilde{X}^{u}_t, \tilde{Y}^{u}_t, \tilde{Z}^{u}_t, \tilde{\hat{\alpha}}^{u}_t).
$$ 
Moreover, $(\tilde{\vartheta}^{u}_t,\tilde{X}^{\prime,u}_t,\tilde{Y}^{\prime,u}_t,\tilde{Z}^{\prime,u}_t,\tilde{X}^u_T, \tilde{\pi}^{u}_t)$ is an independent copy of $(\vartheta^{u}_t,X^{\prime,u}_t,Y^{\prime,u}_t,Z^{\prime,u}_t,X^u_T, \pi^{u}_t)$ which we take to be defined on a separate probability space \( (\tilde{\Omega}, \tilde{\mathcal{F}}, \tilde{\mathbb{P}}) \), and we then work on the product space \( (\Omega \times \tilde{\Omega}, \mathcal{F} \otimes \tilde{\mathcal{F}}, \mathbb{P} \otimes \tilde{\mathbb{P}}) \).
 We shall write $\tilde{\mathbb{E}}$ for the expectation operator on $(\tilde{\Omega},\tilde{\mathcal{{F}}},\tilde{\mathbb{{P}}})$ alone. As usual, we write $(X^\prime,Y^\prime,Z^\prime) = (X^{\prime,u},Y^{\prime,u},Z^{\prime,u})_u$. 

\begin{theorem} \label{Theorem 5.4}
    Suppose Assumption \ref{Ass:5.2} holds for the FBSDE \eqref{eq:4.1}, then the variational equation \eqref{eq::variation} also admits a unique strong solution. 
\end{theorem}
\begin{proof}
It is easy to verify that the variational equation \eqref{eq::variation} satisfies the \ref{ass:2.1.1}, \ref{ASS: bound}, and \ref{ass:2.1.6} of the Assumption \ref{Assumption 4.1}. To simplify the notation, similarly as in \eqref{eq:4.1},  we denote the new coefficients for the variational equation as $\tilde{\phi}^u$ for $\phi = b,\sigma,f, G$, where it is understood that the new coefficients are functions of $(X^{\prime,u},Y^{\prime,u}, Z^{\prime,u}, \pi^u)_u$. Observing that the partial derivative $\tilde{b}^u_{x'} = b^u_x(t, \Theta^u_t) $. By Assumption \ref{Ass:5.2}, we immediately have\[
\frac{1}{2}  (\tilde{b}^u_{x'} + (\tilde{b}^u_{x'})^\top  ) \leq \lambda_1 I_n.
\] The analogous inequality for $\tilde{f}^u_{y'}$ follows similarly. By Remark \ref{Remark:5.1}, the variational equation \eqref{eq::variation} satisfies \ref{ass:2.1.2} of Assumption \ref{Assumption 4.1}.

Using the same notation as in Assumption \ref{Assumption 4.1}, by Lemma \ref{lemma:lipschitz} and the Lipschitz continuity of $b,\sigma,f,G$, we can show that 
\begin{align*}
&|(\tilde{b}^u(t, x, y_1, z_1, r, \eta_1) - \tilde{b}^u(t, x, y_2, z_2, r, \eta_2)| \le \rho_1 |y_1 - y_2| + \rho_2 |z_1 - z_2| \\&
+ \rho_{3} W_{2,m}(\eta_1^{(1)}, \eta_2^{(1)})+\rho_{3} W_{2,m}(\eta_1^{(2)}, \eta_2^{(2)})+\rho_{3} W_{2,m}(\eta_1^{(3)}, \eta_2^{(3)}),\\
&|\tilde{f}^u(t, x_1, y, z_1, r, \eta_1) - \tilde{f}^u(t, x_2, y, z_2, r, \eta_2)| \le \mu_1 |x_1 - x_2|  + \mu_2 |z_1 - z_2| 
\\&
+ \mu_{3} W_{2,m}(\eta_1^{(1)}, \eta_2^{(1)})+\mu_{3} W_{2,m}(\eta_1^{(2)}, \eta_2^{(2)})+\mu_{3} W_{2,m}(\eta_1^{(3)}, \eta_2^{(3)}),\\
&|\tilde{\sigma}^u(t, x_1, y_1, z_1, r, \eta_1) - \tilde{\sigma}^u(t, x_2, y_2, z_2, r,\eta_2)|^2   \leq \hat{w}_1^2 |x_1-x_2|^2 + \hat{w}_2^2 |y_1-y_2|^2 + \\& \quad \hat{w}_3^2 |z_1-z_2|^2  
+ \hat{w}^2_{4} W^2_{2,m}(\eta_1^{(1)}, \eta_2^{(1)})+\hat{w}^2_{4}W^2_{2,m}(\eta_1^{(2)}, \eta_2^{(2)})+\hat{w}^2_{4} W^2_{2,m}(\eta_1^{(3)}, \eta_2^{(3)}),\\
&|\tilde{G}^u(x_1, \zeta_1) - \tilde{G}^u(x_2, \zeta_2)|^2 \le  \hat{\rho}_4^2 |x_1-x_2|^2 + \hat{\rho}_5^2 W_{2,m}^2(\zeta_1, \zeta_2).
\end{align*}
Applying Assumption \ref{Ass:5.2} and Theorem \ref{Theorem 4.4}, we can conclude the proof.
\end{proof}

Under control $\boldsymbol{\hat{\alpha}} + \varepsilon \boldsymbol{\pi}$,  by the result of Theorem \ref{Theorem for FBSDE}, the system of FBSDEs \eqref{eq:4.1} has a unique strong solution and we denote it as $(X^\varepsilon, Y^\varepsilon, Z^\varepsilon)$. We establish the following lemmas.
\begin{lemma} \label{Lm: continuity}
   Let Assumption \ref{Ass:5.2} holds for the FBSDE \eqref{eq:4.1}. Denote  \[
\delta X^u_s := X^{\varepsilon,u}_s - X^u_s ,\quad \delta Y^u_s := Y^{\varepsilon,u}_s - Y^u_s, \quad \delta Z^u_s := Z^{\varepsilon,u}_s - Z^u_s,
\] for $u \in U$, $s \in [0,T]$.  Then, we have
 $$\int \int_0^T \mathbb{E}|\delta X^u_s|^2 \mathrm{d}s\mathrm{d}m(u) \rightarrow 0\quad \text{as}\quad \varepsilon \to 0,$$ and likewise for $Y$ and $Z$.
\end{lemma}
\begin{proof}We set $\|X\|^2_\lambda := \int \int_0^T e^{-\lambda s} \mathbb{E}|X^u_s|^2 \mathrm{d}s\mathrm{d}m(u).$  
We denote \[
\delta X := X^\varepsilon - X, \quad  \delta Y := Y^\varepsilon - Y, \quad  \delta Z := Z^\varepsilon - Z,
\]
By Lemma \ref{Lm:Markov Kernel Lemma}, $\|\delta X\|^2_\lambda $ is well defined, likewise for 
$\|\delta Y\|^2_\lambda, \|\delta Z\|^2_\lambda$.
Set $\bar{\lambda}'_1 = \lambda-2\lambda_1 - C_1^{-1} \rho_1 - C_2^{-1} \rho_2 - (2 +C_3^{-1}+ C_4^{-1}+C_6^{-1}) \rho_{3}-w_1^2 - w_{4}^2-C_5^{-1}\rho_6 $ and  $
\bar{\lambda}'_2 = -\lambda - 2 \lambda_2  - K_1^{-1} \mu_1 - K_2^{-1} (\mu_2 +\mu_3)-(2+ K_3^{-1}+K_5^{-1})\mu_{3} -K_4^{-1}\mu_4,$ where $C_5,C_6,K_4,K_5$ are constants to be specified.
Similarly as the proof in \cite[Theorem 3.1]{CHEN2023105550}, if $\bar{\lambda}_1',\bar{\lambda}_2' > 0$  we have \begin{align*}
\|\delta {Y}\|^2_\lambda 
+  \|\delta {Z}||^2_\lambda  &\leq \tilde{C}_1(
\|\delta {Y}\|^2_\lambda 
+  \|\delta {Z}||^2_\lambda)+\varepsilon^2\tilde{C}_2 \| \boldsymbol{\pi}\|^2_\lambda,
\end{align*}
where 
\begin{align*}
   \tilde{C}_1 =&\bigl[\frac{1  }{{\bar{\lambda}'_{2}}}+ \frac{1}{1 -  K_2 \mu_2 -  K_2 \mu_{3}}\bigr]*
\bigl[
 \rho_4^2 + \rho_5^2 + ( K_1 \mu_1  + K_3 \mu_{3}) \frac{1}{\bar{\lambda}'_{1}}
\bigr]
\\& *
\bigl[
(C_1\rho_1+w_2^2+ C_3\rho_{3}+w_{4}^2) \vee 
  (C_2\rho_2+ w_3^2+ C_4\rho_{3}+w_{4}^2) 
\bigr],
\end{align*}
and $\tilde{C}_2$ is a constant independent of $\varepsilon$. By Assumption \ref{Ass:5.2}, if we choose $C_5$, $C_6$, $K_4$ and $K_5$ large enough,  we  have $\bar{\lambda}'_1, \bar{\lambda}'_2 >0 $ and $\tilde{C}_1 <1$. This implies $
\|\delta {Y}\|^2_\lambda 
+  \|\delta {Z}||^2_\lambda \rightarrow 0$ as $\varepsilon \to 0$, and we can then deduce that
$\|\delta {X}\|^2_\lambda 
\rightarrow 0$ as $\varepsilon \to 0$, which concludes the proof.
\end{proof}

\begin{lemma} \label{lemma:5.5}
Let Assumption \ref{Ass:5.2} hold for the FBSDE \eqref{eq:4.1}. Defining
$$
\Delta X^u := \frac{1}{\epsilon}\delta X^u  - X^{\prime,u}, \quad  \Delta Y^u := \frac{1}{\epsilon} \delta Y^u - Y^{\prime,u}, \quad 
\Delta Z^u := \frac{1}{\epsilon}\delta Z^u - Z^{\prime,u},
$$
we have
$$\int \int_0^T \mathbb{E}|\Delta X^u_s|^2 \mathrm{d}s\mathrm{d}m(u)  \rightarrow 0\quad \text{as}\quad \varepsilon \to 0,$$ and likewise for $Y,Z.$
\end{lemma}
\begin{proof} 
We denote
\begin{align*}
    &\Delta X := \frac{1}{\epsilon}\delta X  - X^\prime, \quad 
\Delta Y := \frac{1}{\epsilon} \delta Y - Y^\prime, \quad 
\Delta Z := \frac{1}{\epsilon}\delta Z - Z^\prime,\\
&\theta^u_{\lambda,\varepsilon} := 
 \bigl(
X^u + \lambda\varepsilon(\Delta X^u + X^{\prime,u}),
Y^u + \lambda\varepsilon(\Delta Y^u + Y^{\prime,u}),
Z^u + \lambda\varepsilon(\Delta Z^u + Z^{\prime,u}),
\hat{\alpha}^u + \lambda\varepsilon \pi^u
 \bigr),\\
 &\tilde{\theta}^u_{\lambda,\varepsilon} := 
 \bigl(
\tilde{X}^u + \lambda\varepsilon(\Delta \tilde{X}^u + \tilde{X}^{\prime,u}),
\tilde{Y}^u + \lambda\varepsilon(\Delta \tilde{Y}^u + \tilde{Y}^{\prime,u}),
\tilde{Z}^u + \lambda\varepsilon(\Delta \tilde{Z}^u + \tilde{Z}^{\prime,u}),
\tilde{\hat{\alpha}}^u + \lambda\varepsilon \tilde{\pi}^u
 \bigr),\\
&\Theta^u_{\lambda,\varepsilon} :=  \bigl(\theta^u_{\lambda,\varepsilon}, (\mathbb{P}_{\theta^{\bar{u}}_{\lambda,\varepsilon}})_{\bar{u}} \bigr).
\end{align*}
 By Lemma \ref{Lm:Markov Kernel Lemma},  $\|\Delta X\|^2_\lambda $ is also well defined, likewise for $\|\Delta Y\|^2_\lambda $ and $\|\Delta Z\|^2_\lambda $.
Recall
\[
\Theta^u_t := \bigl(X^u_t, Y^u_t, Z^u_t, \hat{\alpha}^u_t, (\mathbb{P}^{\bar{u}}_{t})_{\bar{u}}\bigr), \quad \text{and} \quad \vartheta^u_t := (X^u_t, Y^u_t, Z^u_t, \hat{\alpha}^u_t).
\]
For \(\psi = b, \sigma, f\),  \(\tau = x, y, z,\alpha\), and \(i = 1, 2, 3,4\), we let 
\begin{align*}
&A^{u,\varepsilon,\psi}_{\tau,t} := \int_0^1 \psi^u_\tau(t, \Theta^u_{\lambda,\varepsilon, t})  \mathrm{d}\lambda, \quad
\tilde{B}^{u,\tilde{u},\varepsilon,\psi}_{i,t} :=  \int_0^1 \partial_{P^{(i)}}\psi^u(t, \Theta^u_{\lambda,\varepsilon,t})\bigl(\tilde{\theta}^{\tilde{u}}_{\lambda,\varepsilon,t}\bigr) \mathrm{d}\lambda,\\
&C^{u,\varepsilon,\psi}_t :=  
 [A^{u,\varepsilon,\psi}_{x,t} - \psi^u_x(t, \Theta^u_t)] X^{\prime,u}_t 
+ [A^{u,\varepsilon,\psi}_{y,t} - \psi^u_y(t, \Theta^u_t)] Y^{\prime,u}_t 
+ [A^{u,\varepsilon,\psi}_{z,t} - \psi^u_z(t, \Theta^u_t)] Z^{\prime,u}_t \\& \quad +[A^{u,\varepsilon,\psi}_{\alpha,t} - \psi^u_\alpha(t, \Theta^u_t)] \pi^u_t
 + \int \tilde{\mathbb{E}} \bigg[ 
     \bigl( \tilde{B}^{u,\tilde{u},\varepsilon,\psi}_{1,t} - \partial_{P^{(1)}}\psi^u(t, \Theta^u_t)\bigl(\tilde{\vartheta}^{\tilde{u}}_t\bigr)  \bigr) \tilde{X}^{\prime,\tilde{u}}_t 
  + \\& \quad  \bigl( \tilde{B}^{u,\tilde{u},\varepsilon,\psi}_{2,t} - \partial_{P^{(2)}}\psi^u(t, \Theta^u_t)\bigl(\tilde{\vartheta}^{\tilde{u}}_t\bigr)  \bigr) \tilde{Y}^{\prime,\tilde{u}}_t+  \bigl( \tilde{B}^{u,\tilde{u},\varepsilon,\psi}_{3,t} - \partial_{P^{(3)}}\psi^u(t, \Theta^u_t)\bigl(\tilde{\vartheta}^{\tilde{u}}_t\bigr) \bigr) \tilde{ Z}^{\prime,\tilde{u}}_t\\
& \quad +  \bigl( \tilde{B}^{u,\tilde{u},\varepsilon,\psi}_{4,t} - \partial_{P^{(4)}}\psi^u(t, \Theta^u_t)\bigl(\tilde{\vartheta}^{\tilde{u}}_t\bigr)  \bigr) \tilde{\pi}^{\tilde{u}}_t
\bigg] \mathrm{d}m(\tilde{u}),
\end{align*}
\begin{align*}
&D^{u,\varepsilon, G}_T := \int_0^1 G^u_x \Bigl(X^u_T + \lambda \varepsilon (\Delta X^u_T + X^{\prime,u}_T),  (\mathbb{P}_{X^{\tilde{u}}_T + \lambda \varepsilon (\Delta X^{\tilde{u}}_T + X^{\prime,{\tilde{u}}}_T)})_{\tilde{u}} \Bigr)   \mathrm{d}\lambda, \\
&\tilde{F}^{u,\tilde{u},\varepsilon, G}_T :=\!\int_0^1\!\partial_{P^{(1)}}G^u  \bigl(X^u_T + \lambda \varepsilon (\Delta X^u_T + X^{\prime,u}_T),  (\mathbb{P}_{X^{\bar{u}}_T + \lambda \varepsilon (\Delta X^{\bar{u}}_T + X^{\prime,{\bar{u}}}_T)})_{\bar{u}} \bigr)({\tilde{X}^{\tilde{u}}_T + \lambda \varepsilon (\Delta \tilde{X}^{\tilde{u}}_T + \tilde{X}^{\prime,{\tilde{u}}}_T)})\mathrm{d}\lambda, \\
&G^{u,\varepsilon, G}_T  :=   [D^{u,\varepsilon, G}_T - G^u_x(X^u_T, (\mathbb{P}^{\bar{u}}_{T,1})_{\bar{u}}) ] X^{\prime,u}_T
+ \\& \qquad \qquad \int  \tilde{\mathbb{E}}  \Big[  \Bigl( \tilde{F}^{u,\tilde{u},\varepsilon, G}_T - \partial_{P^{(1)}}G^u(X^u_T, (\mathbb{P}^{\bar{u}}_{T,1})_{\bar{u}})\bigl(\tilde{X}^{\tilde{u}}_T\bigr)  \Bigr)  \tilde{X}^{\prime,\tilde{u}}_T  \Big] \mathrm{d}m(\tilde{u}).
\end{align*}
The dynamics of $(\Delta X^u_t,\Delta Y^u_t,\Delta Z^u_t)$ can be rewritten as
\[
\begin{cases}
&\mathrm{d}\Delta X^u_t =
\big\{ A^{u,\varepsilon, b}_{x,t} \Delta X^u_t + A^{u,\varepsilon, b}_{y,t} \Delta Y^u_t + A^{u,\varepsilon, b}_{z,t} \Delta Z^u_t \\
& + \int \tilde{\mathbb{E}} \big[ 
\tilde{B}^{u,\tilde{u},\varepsilon, b}_{1,t} \Delta \tilde{X}^{\tilde{u}}_t 
+ \tilde{B}^{u,\tilde{u},\varepsilon, b}_{2,t} \Delta \tilde{Y}^{\tilde{u}}_t 
+ \tilde{B}^{u,\tilde{u},\varepsilon, b}_{3,t} \Delta \tilde{Z}^{\tilde{u}}_t 
\big] \mathrm{d}m(\tilde{u}) + C^{u,\varepsilon, b}_t \big\} \mathrm{d}t \\
& + \big\{ A^{u,\varepsilon, \sigma}_{x,t} \Delta X^u_t + A^{u,\varepsilon, \sigma}_{y,t} \Delta Y^u_t + A^{u,\varepsilon, \sigma}_{z,t} \Delta Z^u_t \\
& + \int \tilde{\mathbb{E}} \big[
\tilde{B}^{u,\tilde{u},\varepsilon, \sigma}_{1,t} \Delta \tilde{X}^{\tilde{u}}_t 
+ \tilde{B}^{u,\tilde{u},\varepsilon, \sigma}_{2,t} \Delta \tilde{Y}^{\tilde{u}}_t 
+ \tilde{B}^{u,\tilde{u},\varepsilon, \sigma}_{3,t} \Delta \tilde{Z}^{\tilde{u}}_t 
\big] \mathrm{d}m(\tilde{u})+ C^{u,\varepsilon, \sigma}_t \big\} \mathrm{d}B^u_t,\\
&\mathrm{d}\Delta Y^u_t = -\big\{ 
A^{u,\varepsilon, f}_{x,t} \Delta X^u_t + A^{u,\varepsilon, f}_{y,t} \Delta Y^u_t + A^{u,\varepsilon, f}_{z,t} \Delta Z^u_t \\
& + \int \tilde{ \mathbb{E}} \big[
\tilde{B}^{u,\tilde{u},\varepsilon, f}_{1,t} \Delta \tilde{X}^{\tilde{u}}_t 
+ \tilde{B}^{u,\tilde{u},\varepsilon, f}_{2,t} \Delta \tilde{Y}^{\tilde{u}}_t 
+ \tilde{B}^{u,\tilde{u},\varepsilon, f}_{3,t} \Delta \tilde{Z}^{\tilde{u}}_t 
\big] \mathrm{d}m(\tilde{u}) + C^{u,\varepsilon, f}_t \big\} \mathrm{d}t + \Delta Z^u_t  \mathrm{d}B^u_t,
\end{cases}
\]
with initial conditions $\Delta X^u_0 = 0$ and terminal conditions
\[
\Delta Y^u_T = D^{u,\varepsilon, G}_T \Delta X^u_T + \int \tilde{\mathbb{E}} [ \tilde{F}^{u,\tilde{u},\varepsilon, G}_T \Delta \tilde{X}^{\tilde{u}}_T  ] \mathrm{d}m(\tilde{u})+ G^{u,\varepsilon, G}_T.
\]
Here, $(\tilde{\vartheta}^{\tilde{u}}, \tilde{\theta}^u_{\lambda,\varepsilon} ,\tilde{X}^{\prime,\tilde{u}},\tilde{Y}^{\prime,\tilde{u}} ,\tilde{Z}^{\prime,\tilde{u}} ,\tilde{\pi}^{\tilde{u}},\Delta \tilde{X}^{\tilde{u}},\Delta \tilde{Y}^{\tilde{u}},\Delta \tilde{Z}^{\tilde{u}})$ is an independent copy of of the original $(\vartheta^{\tilde{u}},\theta^u_{\lambda,\varepsilon} ,X^{\prime,\tilde{u}},Y^{\prime,\tilde{u}} ,Z^{\prime,\tilde{u}} ,\pi^{\tilde{u}},\Delta X^{\tilde{u}},\Delta Y^{\tilde{u}},\Delta Z^{\tilde{u}})$.

By Lemma~\ref{Lm: continuity}, Assumption~\ref{Ass:5.2} and dominated convergence, we conclude that
\begin{align} \label{converge to 0}
      \|C^{\varepsilon, b}\|_{\lambda}^2 
+ \|C^{\varepsilon,\sigma}\|_{\lambda}^2 
+ \|C^{\varepsilon,f}\|_{\lambda}^2 
+ \int \mathbb{E}\bigl|G^{u,\varepsilon, G}_T\bigr|^2 \mathrm{d}m(u) 
 \rightarrow 0\quad \text{as}\quad \varepsilon \to 0.
\end{align} 
Following the similar argument in Theorem \ref{Theorem 5.4}, we can conclude that $A^{u,\varepsilon,\psi}_{\tau,t} , 
\tilde{B}^{u,\tilde{u},\varepsilon,\psi}_{i,t}$ are still bounded by the corresponding constants in the Assumption \ref{Ass:5.2}. Adopting a similar approach to \cite[Lemmas 2.1--2.2, Remarks 2.1--2.2]{Pardoux} to estimate bounds on $(\|\Delta X\|^2_\lambda,\|\Delta Y\|^2_\lambda,\|\Delta Z\|^2_\lambda )$ and applying a similar approach to Lemma \ref{lemma:5.5} and using  \eqref{converge to 0}, we can conclude the lemma.
\end{proof}

\begin{lemma} \label{lemma 5.6}
If Assumption \ref{Ass:5.2} holds, then we have
\begin{align*}
\int \mathbb{E} \bigg[ \int_0^T \Big\{ 
& \ell^u_x(t, \Theta_t^u) X^{\prime,u}_t + \ell^u_y(t, \Theta^u_t)  Y^{\prime,u}_t + \ell^u_z(t, \Theta^u_t)  Z^{\prime,u}_t + \ell^u_\alpha(t, \Theta^u_t)\pi^u_t   \\
& + \int \tilde{\mathbb{E}} \Big[ \partial_{P^{(1)}}\ell^u(t, \Theta^u_t)\bigl(\tilde{\vartheta}^{\tilde{u}}_t\bigr) \tilde{X}^{\prime,{\tilde{u}}}_t  
+ \partial_{P^{(2)}}\ell^u(t, \Theta^u_t)\bigl(\tilde{\vartheta}^{\tilde{u}}_t\bigr) \tilde{Y}^{\prime,{\tilde{u}}}_t    \\
& +  \partial_{P^{(3)}}\ell^u(t, \Theta^u_t)\bigl(\tilde{\vartheta}^{\tilde{u}}_t\bigr) \tilde{Z}^{\prime,{\tilde{u}}}_t  
+ \partial_{P^{(4)}}\ell^u(t, \Theta^u_t)\bigl(\tilde{\vartheta}^{\tilde{u}}_t\bigr)  \tilde{\pi}^{\tilde{u}}_t \Big] \mathrm{d}m(\tilde{u}) \Big\} \mathrm{d}t  + g^u_y(Y^u_0) Y^{\prime,u}_0 
+ \\
&h^u_x(X^u_T, (\mathbb{P}^{\bar{u}}_{T,1})_{\bar{u}})  X^{\prime,u}_T 
+ \int \tilde{\mathbb{E}} \big[ \partial_{P^{(1)}}h^u\bigl(X^u_T, (\mathbb{P}^{\bar{u}}_{T,1})_{\bar{u}}\bigr)\bigl(\tilde{X}^{\tilde{u}}_T \bigr)  \tilde{X}^{\prime,{\tilde{u}}}_T \big] \mathrm{d}m({\tilde{u}})
\bigg] \mathrm{d}m(u)\geq 0.
\end{align*}
\end{lemma}

\begin{proof}
As \( \boldsymbol{\hat{\alpha}} \) is an optimal control, we have
\[
\frac{1}{\varepsilon}  [ J(\boldsymbol{\hat{\alpha}} + \varepsilon \pi) - J(\boldsymbol{\hat{\alpha}})  ] \geq 0.
\]
Applying Lemma \ref{lemma:5.5}, Assumption \ref{Ass:5.2}, and dominated convergence, we obtain that
\begin{align*}
\frac{1}{\varepsilon} \mathbb{E}  \Big[ h^u\bigl(X^{\varepsilon,u}_T, (\mathbb{P}_{X^{\varepsilon,\bar{u}}_T})_{\bar{u}}\bigr) - h^u\bigl(X^u_T, (\mathbb{P}^{\bar{u}}_{T,1})_{\bar{u}}\bigr)  
\Big] 
& \to \mathbb{E} \bigg[ h^u_x\bigl(X^u_T, (\mathbb{P}^{\bar{u}}_{T,1})_{\bar{u}}\bigr) X^{\prime,u}_T 
+ \\& \int \tilde{\mathbb{E}} \big[ \partial_{P^{(1)}}h^u\bigl(X^u_T, (\mathbb{P}^{\bar{u}}_{T,1})_{\bar{u}}\bigr)\bigl(\tilde{X}^{\tilde{u}}_T\bigr) \tilde{X}^{\prime,{\tilde{u}}}_T ]\mathrm{d}m(\tilde{u})\bigg],
\end{align*}
as \( \varepsilon \to 0 \). By analogous reasoning for the terms $l^u$ and $g^u$ and integrating over the space $U$ on both sides, we conclude the proof.
\end{proof}

Now, let 
\[
\widetilde{\Theta}^u_t := \bigl(\widetilde{X}^u_t, \widetilde{Y}^u_t, \widetilde{Z}^u_t, \tilde{\hat{\alpha}}^u_t, (\mathbb{P}^{\tilde{u}}_{t})_{\bar{u}}\bigr).
\]
We introduce the following adjoint system for our control problem labelled by $u \in U$. 
\begin{equation}\label{adjoint} 
\begin{cases}
d{p}^u_t &= \Big\{ 
    -b_y^{\top,u}(t, \Theta^u_t) q^u_t 
    -\sigma_y^{\top,u}(t, \Theta^u_t) k^u_t 
    +f_y^{\top,u}(t, \Theta^u_t) p^u_t 
    -\ell^u_y(t, \Theta^u_t) + \\
&\quad \int \tilde{\mathbb{E}}  [- \partial_{P^{(2)}}b^{\top,\tilde{u}}(t, \widetilde{\Theta}^{\tilde{u}}_t)\bigl(\vartheta^u_t\bigr) \tilde{q}^{\tilde{u}}_t  
    - \partial_{P^{(2)}}\sigma^{\top,\tilde{u}}(t, \widetilde{\Theta}^{\tilde{u}}_t)\bigl(\vartheta^u_t\bigr) \tilde{k}^{\tilde{u}}_t   \\
&\quad + \partial_{P^{(2)}}f^{\top,\tilde{u}}(t, \widetilde{\Theta}^{\tilde{u}}_t)\bigl(\vartheta^u_t\bigr) \tilde{p}^{\tilde{u}}_t  
    - \partial_{P^{(2)}}\ell^{\tilde{u}}(t, \widetilde{\Theta}^{\tilde{u}}_t)\bigl(\vartheta^u_t\bigr) ] \mathrm{d}m(\tilde{u})
\Big\} \mathrm{d}t \\
&\quad + \Big\{ 
    -b_z^{\top,u}(t, \Theta^u_t) q^u_t 
    -\sigma_z^{\top,u}(t, \Theta^u_t) k^u_t 
    +f_z^{\top,u}(t, \Theta^u_t) p^u_t 
    -\ell^u_z(t, \Theta^u_t) + \\
&\quad \int \tilde{\mathbb{E}}  [ -\partial_{P^{(3)}}b^{\top,\tilde{u}}(t, \widetilde{\Theta}^{\tilde{u}}_t)\bigl(\vartheta^u_t\bigr)\tilde{q}^{\tilde{u}}_t  
    -\partial_{P^{(3)}}\sigma^{\top,\tilde{u}}(t, \widetilde{\Theta}^{\tilde{u}}_t)\bigl(\vartheta^u_t\bigr) \tilde{k}^{\tilde{u}}_t    \\
&\quad + \partial_{P^{(3)}}f^{\top,\tilde{u}}(t, \widetilde{\Theta}^{\tilde{u}}_t)\bigl(\vartheta^u_t\bigr) \tilde{p}^{\tilde{u}}_t  
    -\partial_{P^{(3)}} \ell^{\tilde{u}}(t, \widetilde{\Theta}^{\tilde{u}}_t)\bigl(\vartheta^u_t\bigr) ] \mathrm{d}m(\tilde{u})
\Big\} \mathrm{d}B^u_t,  \\
dq^u_t &= \Big\{ 
    -b_x^{\top,u}(t, \Theta^u_t) q^u_t 
    -\sigma_x^{\top,u}(t, \Theta^u_t) k^u_t 
    +f_x^{\top,u}(t, \Theta^u_t) p^u_t 
    -\ell^u_x(t, \Theta^u_t)+ \\
&\quad \int \tilde{\mathbb{E}}  [- \partial_{P^{(1)}}b^{\top,\tilde{u}}(t, \widetilde{\Theta}^{\tilde{u}}_t)\bigl(\vartheta^u_t\bigr)\tilde{q}^{\tilde{u}}_t  
    -\partial_{P^{(1)}}\sigma^{\top,\tilde{u}}(t, \widetilde{\Theta}^{\tilde{u}}_t)\bigl(\vartheta^u_t\bigr) \tilde{k}^{\tilde{u}}_t   \\
&\quad + \partial_{P^{(1)}}f^{\top,\tilde{u}}(t, \widetilde{\Theta}^{\tilde{u}}_t)\bigl(\vartheta^u_t\bigr) \tilde{p}^{\tilde{u}}_t  
    -\partial_{P^{(1)}}\ell^{\tilde{u}}(t, \widetilde{\Theta}^{\tilde{u}}_t)\bigl(\vartheta^u_t\bigr) ] \mathrm{d}m(\tilde{u})
\Big\} \mathrm{d}t + k^u_t \mathrm{d}B^u_t, 
\end{cases}
\end{equation}
with initial conditions $p^u_0 = -g^u_y(Y^u_0)$ and terminal conditions
\begin{align*}
    q^u_T &= 
    -G_{x}^{\top,u}\bigl(X^u_T, (\mathbb{P}^{\bar{u}}_{T,1})_{\bar{u}}\bigr) p^u_T 
    - \int \tilde{\mathbb{E}}  \bigl[ \partial_{P^{(1)}}G^{\top,\tilde{u}}\bigl(\tilde{X}^{\tilde{u}}_T, (\mathbb{P}^{\bar{u}}_{T,1})_{\bar{u}}\bigr)\bigl(X^u_T\bigr) \tilde{p}^{\tilde{u}}_T  \bigr]\mathrm{d}m(\tilde{u}) \\
&\quad +  h_{x}^{u}\bigl(X^u_T, (\mathbb{P}^{\bar{u}}_{T,1})_{\bar{u}}\bigr) 
    + \int \tilde{\mathbb{E}}  [ \partial_{P^{(1)}}h^{\top,\tilde{u}}\bigl(\tilde{X}^{\tilde{u}}_T, (\mathbb{P}^{\bar{u}}_{T,1})_{\bar{u}}\bigr)\bigl(X^u_T\bigr)   ]\mathrm{d}m(\tilde{u}),
\end{align*}
where $\widetilde{\Theta}^u_t,\tilde{p}^{\tilde{u}},\tilde{q}^{\tilde{u}},\tilde{k}^{\tilde{u}}$ is an independent copy of $\Theta^u_t,p^{\tilde{u}},q^{\tilde{u}},k^{\tilde{u}}$.

To prove the well-posedness of the adjoint equation \eqref{adjoint}, we make the following assumption.

\begin{assumption}\label{Ass:5.3}
We assume there exists positive constants $\bar{\rho}_5$ and $\bar{\mu}_{\gamma,i}$, $\gamma = f,b, \sigma$ and $i = 1,2,3$,  such that for every $u \in U$,
\begin{align*}
    [\int_U \tilde{\mathbb{E}}  [|\partial_{P^{(1)}}G^{\top,\tilde{u}}\bigl(\tilde{X}^{\tilde{u}}_T,(\mathbb{P}_{X^u_T})_u\bigr)\bigl(X^u_T\bigr)|^2]\mathrm{d}m(\tilde{u})]^{\frac{1}{2}} <\bar{\rho}_{5},
\end{align*}
and for every $u \in U$, $t \in [0,T]$ 
\begin{align*}
 [\int_U \tilde{\mathbb{E}}  [|\partial_{P^{(i)}}\gamma^{\top,\tilde{u}}(t,\widetilde{\Theta}_t^{\tilde{u}})(\vartheta^u_t)|^2]\mathrm{d}m(\tilde{u})]^{\frac{1}{2}} <\bar{\mu}_{\gamma,i}.
\end{align*}
We further require 
\begin{align*}
    \int\int_{0}^{T}\mathbb{E}\bigg[ &
| \ell_x^u(t, \Theta^u_t)|^{2} 
+ | \ell_y^u(t, \Theta^u_t)|^{2} +| \ell_z^u(t, \Theta^u_t)|^{2} +| \ell_{\alpha}^u(t, \Theta^u_t)|^{2} + \\&\sum_{i=1}^4 \int\tilde{\mathbb{E}}[|\partial_{P^{(i)}}\ell^{\tilde{u}}(t, \widetilde{\Theta}^{\tilde{u}}_t)\bigl(\vartheta^u_t\bigr)|^2] \mathrm{d}m(\tilde{u})\bigg] 
\mathrm{d}t\mathrm{d}m(u) < \infty,
\end{align*}
and 
\begin{align*}
    \int\mathbb{E}\bigg[ 
\big|h^u_x\bigl(X^u_T, (\mathbb{P}^{\bar{u}}_{T,1})_{\bar{u}}\bigr)\big|^{2} 
+ \big| g^u_y(Y^u_0)\big|^{2}+   \int \tilde{\mathbb{E}}  [ |\partial_{P^{(1)}}h^{\tilde{u}}\bigl(\tilde{X}^{\tilde{u}}_T, (\mathbb{P}^{\bar{u}}_{T,1})_{\bar{u}}\bigr)\bigl(X^u_T\bigr)   |^2]\mathrm{d}m(\tilde{u})\bigg] 
\mathrm{d}m(u) < \infty.
\end{align*}
Finally, we assume 
    $$
    2\lambda_2+2\lambda_1 < - 2\bar{\mu}_{f,2}  - 6\mu_2^2 - 6\bar{\mu}_{f,3}^2  - 2\bar{\mu}_{b,1}-  (\bar{\mu}_{\sigma,1}+w_1)^2. 
    $$
\end{assumption}

We denote the coefficients and the terminal conditions of the adjoint equation \eqref{adjoint} by 
$\hat{b}^u, \hat{\sigma}^u, \hat{f}^u, \hat{G}^u$, labelled by $u \in U$. To emphasize the structural condition \eqref{adjoint} satisfies, we will omit its explicit dependence on $(X^u_t,Y^u_t,Z^u_t,\hat{\alpha}^u_t)_u$, which together can be treated as $(\Lambda^u)_u$ as in \eqref{eq:3.2}. Suppose Assumption \ref{Ass:5.2} and \ref{Ass:5.3} holds, we now use the notation \( p, p_1, p_2 \in \mathbb{R}^l \), \( q, q_1, q_2 \in \mathbb{R}^n \), \( k, k_1, k_2 \in \mathbb{R}^{n \times d} \),  \( \zeta_1, \zeta_2 \in \mathcal{P}^2_m(\mathbb{R}^l) \),  and \( \xi, \eta_1, \eta_2 \in \mathcal{P}^2_m(\mathbb{R}^l \times \mathbb{R}^n \times \mathbb{R}^{n \times d}) \). Denote $\Delta \varphi \triangleq \varphi_1 -\varphi_2$, for $\varphi = p,q,k$. For all choices of the aforementioned variables, we can show
    \begin{align*}
\langle \hat{b}^u(t, p_1, q, k,\xi) - \hat{b}^u(t, p_2, q, k, \xi), \Delta p \rangle &\leq \lambda_2 |\Delta p|^2, \\
\langle \hat{f}^u(t, p, q_1, k, \xi) - \hat{f}^u(t, p, q_2, k, \xi), \Delta q \rangle &\leq \lambda_1 |\Delta q|^2. 
\end{align*}

Applying the second half of the result in Lemma \ref{lemma:lipschitz}, we also have
\begin{align*}
|\hat{b}^u(t, p, q_1, k_1, \eta_1) - \hat{b}^u(t, p, q_2, k_2, \eta_2)| &\leq \rho_1 |\Delta q| + w_2 |\Delta k| + \bar{\mu}_{f,2} W_{2,m}(\eta_1^{(1)}, \eta_2^{(1)}) \\
&\quad + \bar{\mu}_{b,2} W_{2,m}(\eta_1^{(2)}, \eta_2^{(2)}) + \bar{\mu}_{\sigma,2} W_{2,m}(\eta_1^{(3)}, \eta_2^{(3)}), \\
|\hat{f}^u(t, p_1, q, k_1, \eta_1) - \hat{f}^u(t, p_2, q, k_2, \eta_2)| &\leq \mu_1 |\Delta p| + w_1 |\Delta k| + \bar{\mu}_{f,1} W_{2,m}(\eta_1^{(1)}, \eta_2^{(1)}) \\
&\quad + \bar{\mu}_{b,1} W_{2,m}(\eta_1^{(2)}, \eta_2^{(2)}) + \bar{\mu}_{\sigma,1} W_{2,m}(\eta_1^{(3)}, \eta_2^{(3)}), \\
|\hat{\sigma}^u(t, p_1, q_1, k_1, \eta_1) - \hat{\sigma}^u(t, p_2, q_2, k_2, \eta_2)|^2 &\leq 6\mu_2^2 |\Delta p|^2 + 6\rho_2^2 |\Delta q|^2 + \hat{w}_3^2 |\Delta k|^2 \\
&\quad + 6\bar{\mu}_{f,3}^2 W_{2,m}^2(\eta_1^{(1)}, \eta_2^{(1)}) + 6\bar{\mu}_{b,3}^2 W_{2,m}^2(\eta_1^{(2)}, \eta_2^{(2)}) + \\& \qquad 6\bar{\mu}_{\sigma,3}^2W_{2,m}^2(\eta_1^{(3)}, \eta_2^{(3)}), \\
|\hat{G}^u(p_1, \zeta_1) - \hat{G}^u(p_2, \zeta_2)|^2 &\leq \hat{\rho}_4^2 |\Delta p|^2 + 2\bar{\rho}_5^2 W_{2,m}^2(\zeta_1, \zeta_2).
\end{align*}

Set
\begin{align*}
\bar{\lambda}_{21} &= \bar{\lambda} - 2\lambda_2 - \bar{C}_1^{-1} \rho_1 - \bar{C}_2^{-1} w_2 - 2\bar{\mu}_{f,2} - \bar{C}_3^{-1} \bar{\mu}_{b,2} - \bar{C}_4^{-1} \bar{\mu}_{\sigma,2}  - 6\mu_2^2 - 6\bar{\mu}_{f,3}^2, \\
\bar{\lambda}_{22} &= -\bar{\lambda} - 2\lambda_1 - \bar{K}_1^{-1} \mu_1 - \bar{K}_2^{-1} w_1 - 2\bar{\mu}_{b,1} - \bar{K}_3^{-1} \bar{\mu}_{f,1}- \bar{K}_2^{-1} \bar{\mu}_{\sigma,1}.
\end{align*}

By Theorem \ref{Theorem 4.4}, we have the following well-posedness result for the adjoint equation \eqref{adjoint}.
\begin{theorem} \label{Theorem for adjoint}
Assume Assumption \ref{Ass:5.2} and Assumption \ref{Ass:5.3} hold. We choose the appropriate $\bar{\lambda}$, $0<\bar{K}_2< \frac{1}{w_1+\bar{\mu}_{\sigma,1}}$, and choose $\bar{C}_1, \bar{C}_2,\bar{C}_3,\bar{C}_4,\bar{K}_1,\bar{K}_3$ sufficiently large such that $\bar{\lambda}_{21}$ and $\bar{\lambda}_{22}$ are positive. Then there exists a positive constant $$ \theta_{2} = \frac{1}{\bigl( \frac{1}{\bar{\lambda}_{22}} + \frac{1}{1 -\bar{K}_2 w_1 -\bar{K}_2\bar{\mu}_{\sigma,1}} \bigr)
\times \bigl[ \hat{\rho}^2_4 + 2\bar{\rho}_5^2 + \frac{\bar{K}_1 \mu_1 + \bar{K}_3 \bar{\mu}_{f,1}}{\bar{\lambda}_{21}} \bigr]},$$ which does not depend on T. Provided this holds, if 
\[
\bar{C}_1\rho_1+6\rho_2^2+ \bar{C}_3\bar{\mu}_{b,2}+ 6\bar{\mu}_{b,3}^2  \vee
\bar{C}_2w_2+ \hat{w}_3^2+ \bar{C}_4\bar{\mu}_{\sigma,2}+6\bar{\mu}_{\sigma,3}^2 \in [0, \theta_{2}),
\]  
then the adjoint equation \eqref{adjoint} admits a unique strong solution where $(X^u_t,Y^u_t,Z^u_t,\hat{\alpha}^u_t)_u$ together being treated as $(\Lambda^u)_u$ as in \eqref{eq:3.2}. 
\end{theorem}
\begin{assumption} \label{Assumption for total}
    Let $\theta'_1 = \theta_{1} \wedge \theta_{2}$, we assume \[\bar{C}_1\rho_1+6\rho_2^2+ \bar{C}_3\bar{\mu}_{b,2}+ 6\bar{\mu}_{b,3}^2, 
\bar{C}_2w_2+ \hat{w}_3^2+ \bar{C}_4\bar{\mu}_{\sigma,2}+6\bar{\mu}_{\sigma,3}^2,
C_1\rho_1+\hat{w}_2^2+ C_3\rho_{3}+\hat{w}_{4}^2,
C_2\rho_2+ \hat{w}_3^2+ C_4\rho_{3}+\hat{w}_{4}^2
\] are all smaller than $\theta'_1$.
\end{assumption}
By combining Theorem \ref{Theorem for FBSDE}, Theorem \ref{Theorem 5.4}, and Theorem \ref{Theorem for adjoint}, we have the following theorem.
\begin{theorem}\label{thm:FBSDE_variational_adjoint}
    If Assumptions \ref{Ass:5.2}--\ref{Assumption for total} hold, then the FBSDE \eqref{eq:4.1}, variational equation \eqref{eq::variation}, and the adjoint equation \eqref{adjoint} all admit a unique strong solution.
\end{theorem}
To derive the stochastic maximum principle, we define the generalised Hamiltonian 
\begin{align*}
H^u(t, X, Y, Z, \alpha^u, \xi, p, q, k) &
:= \langle q^u_t, b^u(t, X^u, Y^u, Z^u, \alpha^u, \xi) \rangle  
+ \langle k^u_t, \sigma^u(t, X^u, Y^u, Z^u, \alpha^u, \xi) \rangle \\
& - \langle p^u_t, f^u(t, X^u, Y^u, Z^u, \alpha^u, \xi) \rangle 
+ \ell^u(t, X^u, Y^u, Z^u, \alpha^u, \xi) 
\end{align*}
for each $u \in U$. We shall need the following assumption on the Lions derivative.

\begin{assumption} \label{assum: finite for alpha}
     There exists a constant $C_{\alpha}$ such that, for every $u \in U$ and $t \in [0,T]$,
\begin{align*}
 [\int_U \tilde{\mathbb{E}}  [|\partial_{P^{(4)}}\gamma^{\top,\tilde{u}}(t,\widetilde{\Theta}_t^{\tilde{u}})(\vartheta^u_t)|^2]\mathrm{d}m(\tilde{u})]^{\frac{1}{2}} <C_{\alpha},
\end{align*} 
where $\gamma = f,b, \sigma$.
\end{assumption}
\begin{theorem}[Maximum principle]\label{Tm:maximum principle}
Suppose Assumptions \ref{Ass:5.2} -- \ref{assum: finite for alpha} hold. Let \( \boldsymbol{\hat{\alpha}} := (\hat{\alpha}^u)_u \) be an optimal control with corresponding trajectory \( (X, Y, Z) \). For any $\boldsymbol{v} := (v^u)_u \in \mathcal{A}_{ad}$, it holds for $m$-a.e.~$u\in U$ that
\begin{align*}
 &\Bigl\langle H^u_{\alpha}(t, X, Y, Z, \hat{\alpha}^u, \xi, p, q, k) 
\,+  \\& \qquad \qquad \int \tilde{\mathbb{E}}[ \partial_{P^{(4)}}H^{\tilde{u}}(t, \tilde{X}, \tilde{Y}, \tilde{Z}, \tilde{\hat{\alpha}}^u, \xi, \tilde{p}, \tilde{q}, \tilde{k})\bigl(\vartheta^u_t\bigr)] \mathrm{d}m(\tilde{u}), 
\,v^u_t -\hat{\alpha}^u_t \Bigr\rangle  \geq 0,
\end{align*}
 a.s., for a.e.~$t\in[0,T]$, where \( (p, q, k) \) is the strong solution of the adjoint equation \eqref{adjoint}.    
\end{theorem}
\begin{proof}
We apply Itô’s formula on $\langle p^u_t, Y^{',u}_t\rangle  + \langle q^u_t, X^{',u}_t\rangle$ and integrate against $u$. After rearranging the terms, we can conclude 
\begin{align*}
&\int \mathbb{E}\Big[g_y^u(Y_0^u)Y^{\prime,u}_0+  \\& h_{x}^{u}\bigl(X^u_T, (\mathbb{P}^{\bar{u}}_{T,1})_{\bar{u}}\bigr)X^{\prime,u}_T 
    +  \int \tilde{\mathbb{E}} [ \partial_{P^{(1)}}h^{\top,\tilde{u}}\bigl(\tilde{X}^{\tilde{u}}_T, (\mathbb{P}^{\bar{u}}_{T,1})_{\bar{u}}\bigr)\bigl(X^u_T\bigr) *X^{\prime,u}_T ]\mathrm{d}m(\tilde{u})\Big] \mathrm{d}m(u)
 \\&  = \int \int_0^T \mathbb{E}\Big[Y^{\prime,u}_t \bigl( 
    -\ell^u_y(t, \Theta^u_t) -\int 
    \tilde{\mathbb{E}}  [ \partial_{P^{(2)}}\ell^{\tilde{u}}(t, \widetilde{\Theta}^{\tilde{u}}_t)(\vartheta^u_t) ] \mathrm{d}m(\tilde{u}) \bigr) -p_t^u \bigl(
  f^u_{\alpha}(t, \Theta^u_t) \pi^u_t + \\&\int \tilde{\mathbb{E}}[ \partial_{P^{(4)}}f^u(t, \Theta^u_t)(\tilde{\vartheta}^{\tilde{u}}_t) \tilde{\pi}_t^{\tilde{u}} ] \mathrm{d}m(\tilde{u})\bigr )+  Z^{\prime,u}_t \bigl( 
    -\ell^u_z(t, \Theta^u_t)
    - \int \tilde{\mathbb{E}} [ \partial_{P^{(3)}}l^{\tilde{u}}(t, \widetilde{\Theta}^{\tilde{u}}_t)(\vartheta^u_t)] \mathrm{d}m(\tilde{u}) \bigr) \\& +q^u_t\bigl( b^u_{\alpha}(t, \Theta^u_t) \pi^u_t +\int \tilde{\mathbb{E}}[ \partial_{P^{(4)}}b^u(t, \Theta^u_t)(\tilde{\vartheta}^{\tilde{u}}_t) \tilde{\pi}_t^{\tilde{u}} ] \mathrm{d}m(\tilde{u})\bigr)+  X^{\prime,u}_t  \bigl( 
    -\ell^u_x(t, \Theta^u_t) 
    -\\&\int \tilde{\mathbb{E}}[ \partial_{P^{(1)}}\ell^{\tilde{u}}(t, \widetilde{\Theta}^{\tilde{u}}_t)(\vartheta^u_t)] \mathrm{d}m(\tilde{u}) \bigr) + k^u_t \bigl(\sigma^u_{\alpha}(t, \Theta^u_t) \pi^u_t 
    + \int \tilde{\mathbb{E}} [ \partial_{P^{(4)}}\sigma^u(t, \Theta^u_t)(\tilde{\vartheta}^{\tilde{u}}_t) \tilde{\pi}_t^{\tilde{u}}  ] \mathrm{d}m(\tilde{u}) \bigr)\Big] \mathrm{d}t \mathrm{d}m(u)
\end{align*}

Combining with the result in Lemma \ref{lemma 5.6}, we have 
\begin{align*}
\int \mathbb{E} \bigg[ \int_0^T 
&  \ell^u_\alpha(t, \Theta^u_t) \pi^u_t + \int \tilde{\mathbb{E}}  [ \partial_{P^{(4)}}\ell^u(t, \Theta^u_t)(\tilde{\vartheta}^{\tilde{u}}_t) \tilde{\pi}^{\tilde{u}}_t  ] \mathrm{d}m(\tilde{u}) -\\&p_t^u \bigl(
  f^u_{\alpha}(t, \Theta^u_t) \pi^u_t +\int \tilde{\mathbb{E}} [ \partial_{P^{(4)}}f^u(t, \Theta^u_t)(\tilde{\vartheta}^{\tilde{u}}_t) \tilde{\pi}_t^{\tilde{u}}  ] \mathrm{d}m(\tilde{u}) \bigr)+
 \\&q^u_t \bigl( b^u_{\alpha}(t, \Theta^u_t) \pi^u_t +\int \tilde{\mathbb{E}} [ \partial_{P^{(4)}}b^u(t, \Theta^u_t)(\tilde{\vartheta}^{\tilde{u}}_t) \tilde{\pi}_t^{\tilde{u}}  ] \mathrm{d}m(\tilde{u})\bigr)+ \\& k^u_t \bigl(\sigma^u_{\alpha}(t, \Theta^u_t) \pi^u_t 
    + \int \tilde{\mathbb{E}} [ \partial_{P^{(4)}}\sigma^u(t, \Theta^u_t)(\tilde{\vartheta}^{\tilde{u}}_t) \tilde{\pi}_t^{\tilde{u}}  ] \mathrm{d}m(\tilde{u}) \bigr)\mathrm{d}t 
\bigg] \mathrm{d}m(u)\geq 0.
\end{align*}

Applying Fubini's theorem and rearranging the terms, we can conclude that, for any $\boldsymbol{\pi}$  such that  $\boldsymbol{\hat{\alpha}}+ \boldsymbol{\pi} \in \mathcal{A}_{ad}$, we have that
\begin{align*}
    \int_0^T  \int  \mathbb{E} \Big[&
 \Bigl\langle H^u_{\alpha}(t, X, Y, Z, \hat{\alpha}^u, \xi, p, q, k) 
\\&+  \int \tilde{\mathbb{E}}[ \partial_{P^{(4)}}H^{\tilde{u}}(t, \tilde{X}, \tilde{Y}, \tilde{Z}, \tilde{\hat{\alpha}}^u, \xi, \tilde{p}, \tilde{q}, \tilde{k})\bigl(\vartheta^u_t\bigr)] \mathrm{d}m(\tilde{u}), 
\pi^u_t  \Bigr\rangle\Big]\mathrm{d}m(u) \mathrm{d}t
\end{align*}
is always non-negative. We thus conclude the proof.
\end{proof}
\subsection{Verification Theorem}\label{sect:verification}
\begin{assumption} \label{Ass:convex}
Let $(X, Y, Z, \alpha), (X', Y', Z', \alpha')$ live in  $E:= \mathbb{R}^n \times \mathbb{R}^l \times \mathbb{R}^{l \times d} \times \mathbb{R}^k$, and $\zeta, \zeta' \in \mathcal{P}^2_m(\mathbb{R}^n)$, $\xi, \xi' \in \mathcal{P}^2_m(E)$. Let $\bar{\tilde{X}}, \bar{\tilde{X}}'$, $(\tilde{X}, \tilde{Y}, \tilde{Z}, \tilde{A}),$ and $ (\tilde{X}', \tilde{Y}', \tilde{Z}', \tilde{A}')$ be defined on $(\tilde{\Omega},\mathcal{\tilde{F}}, \tilde{\mathbb{P}})$ such that $\tilde{\mathbb{P}}_{\bar{\tilde{X}}} = \zeta$, $\tilde{\mathbb{P}}_{\bar{\tilde{X}}'} = \zeta'$, $\tilde{\mathbb{P}}_{(\tilde{X}, \tilde{Y}, \tilde{Z}, \tilde{A})} = \xi$,  and $\tilde{\mathbb{P}}_{(\tilde{X}', \tilde{Y}', \tilde{Z}', \tilde{A}')} = \xi'$, respectively.
Denote
\begin{align*}
\Theta &:= (X,Y, Z, \alpha, \xi, p, q, k), \quad \Theta' := (X',Y', Z', \alpha', \xi', p, q, k), \\
\vartheta^u &:= (X^u,Y^u, Z^u, A^u), \quad \tilde{\vartheta}^u := (\tilde{X}^u, \tilde{Y}^u, \tilde{Z}^u, \tilde{A}^u).
\end{align*}
The following conditions must hold for all values of the above variables:

\begin{itemize}
    \item[(i)] For $m$-a.e.~$u \in U$, we have
    \begin{align*}
     h^u(X, \zeta)- h^u(X^{'}, \zeta') &\leq h^u_x(X, \zeta)  (  X-X^{'}) +  \\& \int \tilde{\mathbb{E}} [ \partial_{P^{(1)}}h^u(X, \zeta)\bigl(\bar{\tilde{X}}^{\tilde{u}}\bigr)(  \bar{\tilde{X}}^{\tilde{u}}-\bar{\tilde{X}}^{\tilde{u},'})  ] \mathrm{d}m(\tilde{u}), \\
    g^u(Y)- g^u(Y^{'})  &\leq g^u_y(Y)  (Y-Y^{'} ).
    \end{align*}
    \item[(ii)] For $m$-a.e.~$u \in U$, all $t \in [0,T]$, we have
    \begin{align*}
    & H^u(t, \Theta) - H^u(t, \Theta') \leq
    \\& \quad H^u_{x}(t, \Theta) (X-X^{'})
     + H^u_{y}(t, \Theta)(Y-Y^{'} )
     +  H^u_{z}(t, \Theta)  (Z-Z^{'} )
     + H^u_{\alpha}(t, \Theta) (\alpha-\alpha' ) + \\& \quad\int \tilde{\mathbb{E}} \Big[
        \partial_{P^{(1)}}H^u(t, \Theta)\bigl(\tilde{\vartheta}^{\tilde{u}}\bigr)  (\tilde{X}^{\tilde{u}} -\tilde{X}^{\tilde{u},'} ) 
        + \partial_{P^{(2)}}H^u(t, \Theta)\bigl(\tilde{\vartheta}^{\tilde{u}}\bigr)  (\tilde{Y}^{\tilde{u}}-\tilde{Y}^{\tilde{u},'} ) \\
    & \quad+ \partial_{P^{(3)}}H^u(t, \Theta)\bigl(\tilde{\vartheta}^{\tilde{u}}\bigr)  (\tilde{Z}^{\tilde{u}}-\tilde{Z}^{{\tilde{u}},'} )
        + \partial_{P^{(4)}}H^u(t, \Theta)\bigl(\tilde{\vartheta}^{\tilde{u}}\bigr) (\tilde{A}^{\tilde{u}}-\tilde{A}^{{\tilde{u}},'} ) \Big] \mathrm{d}m(\tilde{u}).
    \end{align*}
    
    \item[(iii)] For $m$-a.e.~$u \in U$,
    \begin{align*}
    G^u(X, \zeta)-G^u(X', \zeta') 
    =&G^u_x(X, \zeta)(X-X' ) 
    + \\&\int \tilde{\mathbb{E}} [ \partial_{P^{(1)}}G^u(X, \zeta)\bigl(\bar{\tilde{X}}^{\tilde{u}}\bigr) (\bar{\tilde{X}}^{\tilde{u}}-\bar{\tilde{X}}^{\tilde{u},'} )  ] \mathrm{d}m(\tilde{u}).
    \end{align*}
\end{itemize}
\end{assumption}
\begin{theorem}[Verification theorem]\label{thm:verification}
    Let Assumptions \ref{Ass:5.2}–\ref{Ass:convex} hold. Let $\hat{\alpha}$ be an admissible control with corresponding trajectory $(X, Y, Z)$. Furthermore, let $(p, q, k)$ be the solution of the adjoint equation \eqref{adjoint} for the control $\hat{\alpha}$. If, for any $v:= (v^u)_u\in \mathcal{A}_{ad}$, we have
\begin{align} 
 &\bigl\langle H^u_{\alpha}(t, X, Y, Z, \hat{\alpha}^u, \xi, p, q, k) 
\,+ \notag \\& \label{maximum condition}  \qquad\qquad\int \tilde{\mathbb{E}}[ \partial_{P^{(4)}}H^{\tilde{u}}(t, \tilde{X}, \tilde{Y}, \tilde{Z}, \tilde{\hat{\alpha}}^u, \xi, \tilde{p}, \tilde{q}, \tilde{k})\bigl(\vartheta^u_t\bigr)] \mathrm{d}m(\tilde{u}),\,
v^u_t -\hat{\alpha}^u_t \bigr\rangle  \geq 0,
\end{align}
a.s., for a.e.~$t\in[0,T]$ and for $m$-a.e.~$u\in U$, then $\hat{\alpha}$ is an optimal control.
\end{theorem}
\begin{proof}
    Let $\alpha'$ be any other admissible control with corresponding dynamics $(X', Y', Z')$. Set
\[
\Delta \varphi^u := \varphi^u - \varphi^{u,'}, \quad \text{for } \varphi = X, Y, Z,
\]
From Assumption \ref{Ass:convex}, we have
\begin{align*}
&\int h^u\bigl(X^u_T, (\mathbb{P}^{\bar{u}}_{T,1})_{\bar{u}}\bigr) - h^u\bigl(X^{u,'}_T, (\mathbb{P}_{X^{\bar{u},'}_T})_{\bar{u}}\bigr) \mathrm{d}m(u)
\leq \int \Big[ h^u_x\bigl(X^u_T, (\mathbb{P}^{\bar{u}}_{T,1})_{\bar{u}}\bigr)  \Delta X^u_T 
+ \\& \qquad \int \tilde{\mathbb{E}} [ \partial_{P^{(1)}}h^u\bigl(X^u_T, (\mathbb{P}^{\bar{u}}_{T,1})_{\bar{u}}\bigr)\bigl(\tilde{X}^{\tilde{u}}_T\bigr)  \Delta \tilde{X}^{\tilde{u}}_T  ] \mathrm{d}m(\tilde{u})\Big]\mathrm{d}m(u), \\
&\int g^u(Y^u_0) - g^u(Y^{u,'}_0) \mathrm{d}m(u)
\leq \int g^u_y(Y_0)  \Delta Y^u_0 \mathrm{d}m(u).
\end{align*}
Applying Itô’s formula to $\int q^u_t \Delta X^u_t+p^u_t \Delta Y^u_t \mathrm{d}m(u)$ and rearranging the terms, we get
\begin{align*}
J(\hat{\alpha}) - J(\alpha') \leq \int& \int_0^T  \mathbb{E}  \biggl[ 
     H^u(t, \Theta_t) - H^u(t, \Theta_t')- H^u_{x}(t, \Theta) \Delta X^u_t
     - H^u_{y}(t, \Theta) \Delta Y^u_t
     - \\&  H^u_{z}(t, \Theta) \Delta Z^u_t -
      \Big(\int \tilde{\mathbb{E}} \bigl[
       \partial_{P^{(1)}} H^u(t, \Theta)\bigl(\tilde{\vartheta}^{\tilde{u}}\bigr)\Delta \tilde{X}^{\tilde{u}}_t
        + \partial_{P^{(2)}}H^u(t, \Theta)\bigl(\tilde{\vartheta}^{\tilde{u}}\bigr)  \Delta \tilde{Y}^{\tilde{u}}_t \\& +  \partial_{P^{(3)}}H^u(t, \Theta)\bigl(\tilde{\vartheta}^{\tilde{u}}\bigr)  \Delta \tilde{Z}^{\tilde{u}}_t
         \bigr] \mathrm{d}m(\tilde{u})\Big)\biggr]\mathrm{d}t\mathrm{d}m(u).
\end{align*}
Therefore, by Assumption \ref{Ass:convex}, rearranging the terms and invoking \eqref{maximum condition}, we conclude that $\hat{\alpha}$ is indeed the optimal control.
\end{proof}
\section{Well-posedness of the FBSDE system}\label{sect:well-posedness}

In the first part of this section, we construct a suitable path space on which to work with the laws of solutions to our FBSDEs. First of all, we set
\begin{equation}\label{eq:S1S2}
S_1 := C([0,T],\mathbb{R}^n) \times C([0,T],\mathbb{R}^{l}) \times L^2([0,T],\mathbb{R}^{l\times d}) \times  L^2([0,T],\mathbb{R}^{k})
\end{equation}
and $S_2:= C([0,T],\mathbb{R}^d)\times \mathbb{R}^n$, which we equip with their Borel $\sigma$-algebras. Crucially, the spaces are Polish, with the Borel $\sigma$-algebra on $C$ generated by the coordinate projections $h\mapsto h(t)$ and the one on $L^2$ generated by the continuous linear functionals $h\mapsto \int_0^T h(s)g(s)\text{d}s$.

Some care is needed for two reasons: the elements of the $L^2$ spaces are equivalence classes and their  measurability in time is with respect to the Lebesgue $\sigma$-algebra. This is resolved by working with precise representatives as follows. Using \cite[Theorem 1.34]{evans}, we can uniquely identify any given equivalence class $[z]\in L^2([0,T],\mathbb{R}^{l\times d})$ with $\bar{z}\in[z]$ defined by
\begin{equation}\label{eq:L2_time_project}
 \bar{z}(t) := \hat{z}(t)\mathbf{1}_{\hat{z}(t)\in \mathbb{R}},\;\;\text{where} \;\;\hat{z}(t)=\liminf_{\varepsilon \downarrow 0} z^\varepsilon(t),\;\;z^\varepsilon(t)=\frac{1}{\varepsilon}\int_{[(t-\varepsilon)\lor 0,t]} z(s) \text{d}s.
\end{equation}
We can then let the canonical process $z_\cdot$ on $L^2([0,T],\mathbb{R}^{l\times d})$ be given by $(t,[z])\mapsto z_t := \bar{z}(t)$. Now consider the natural filtration $\sigma(z_s:s\leq t)$.  One readily confirms that each $(t,[z])\mapsto z^\varepsilon(t)$ is predictable for this filtration and so $(t,[z])\mapsto\hat{z}(t)$ is predictable as a process with values in the extended reals. In particular, the event $\{(t,[z]):\hat{z}(t) \in \mathbb{R}\}$ belongs to the predictable $\sigma$-algebra, so we can finally conclude that $z_\cdot$ is predictable. Analogously, we introduce a canonical process $\lambda_\cdot$ on $L^2([0,T],\mathbb{R}^k)$. Then, the canonical process $(x_\cdot,y_\cdot,z_\cdot,\lambda_\cdot)$ on $(S_1,\mathcal{B}(S_1))$, given by
\begin{equation}\label{eq:canonical_process}
(x,y,[z],[\lambda])\mapsto (x_t,y_t,z_t,\lambda_t) :=  (x(t),y(t),\bar{z}(t),\bar{\lambda}(t)),
\end{equation}
is well-defined and yields a predictable process for the natural filtration
$\sigma(x_s,y_s,z_s,\lambda_s:s\leq t)$.

 When working on $\bar{\Omega}:= S_1 \times S_2$ with its Borel sigma algebra $\bar{\mathcal{{F}}}:= \mathcal{B}(S_1\times S_2 )$, we write $(w_\cdot, \chi)$ for the canonical process on $S_2$ and extend all the processes to $\bar{\Omega}$ in the obvious way.

\begin{remark}\label{rem:path_space_laws} Coming back to our FBSDE system \eqref{eq:3.2}, the Markov kernel $(\mathbb{P}^u)_u$ in Definition \ref{def:strong_soln} is realised as $\mathbb{P}^u:= \mathbb{P}\circ(X^u,Y^u,Z^u,\Lambda^u)^{-1}$ on the path space $(S_1,\mathcal{B}(S_1))$. This of course requires that $(X^u,Y^u,Z^u,\Lambda^u):\Omega \rightarrow S_1$ is a measurable map which is taken care of as part of the well-posedness arguments in this section.
\end{remark}

\subsection{Auxiliary FBSDE problem}

The exact setup introduced above will play a crucial role in the proofs of this section. For now, let $(\hat{\Omega},\hat{\mathcal{F}},\hat{\mathbb{P}})$ be some other complete probability space supporting a $d$-dimensional Brownian motion $\hat{B}$ and random variables $\hat{\chi}_0:\hat{\Omega} \rightarrow \mathbb{R}^n$ and 
$\mathfrak{u}:\hat{\Omega} \rightarrow  U$ that are (i) independent of $\hat{B}$ and (ii) satisfy that the conditional law of $\hat{\chi}_0$ given $\mathfrak{u}$ is the Markov kernel $(\text{Law}(\chi_0^u))_u$. Let $\hat{\mathbb{F}} := \{\hat{\mathcal{F}}_t\}_{t \in [0,T]}$ be the filtration generated by $(\hat{B},\hat{\chi}_0, \mathfrak{u})$, augmented with the $\hat{\mathbb{P}}$-null sets of $\hat{\mathcal{F}}$, and define the space $\hat{\mathcal{H}}^2(\mathbb{R}^n \times \mathbb{R}^l \times \mathbb{R}^{l \times d} \times \mathbb{R}^k) $ consisting of quadruple of $\hat{\mathbb{F}}$-progressively measurable processes $(\hat{X},
\hat{Y}, \hat{Z}, \hat{\Lambda})$ with values in $\mathbb{R}^n \times \mathbb{R}^l \times \mathbb{R}^{l \times d} \times \mathbb{R}^{k}$ so that
\[
\hat{\mathbb{E}}[\sup_{t \in [0,T]}|\hat{X}_t|^2]+\hat{\mathbb{E}}[\sup_{t \in [0,T]}|\hat{Y}_t|^2]
+ \int_0^T \hat{\mathbb{E}}[|\hat{Z}_t |^2 ]\mathrm{d}t+\int_0^T \hat{\mathbb{E}}[|\hat{\Lambda}_t |^2 ]\mathrm{d}t<\infty.
\] 
The central idea of this section is to `lift' our problem to the analysis of a single mean-field FBSDE, which is randomised according to the distribution $m$ on the index space $U$. That is, we look for a solution $(\hat{X}, \hat{Y}, \hat{Z}, \hat{\Lambda}) \in \hat{\mathcal{H}}^2(\mathbb{R}^n \times \mathbb{R}^l \times \mathbb{R}^{l \times d}\times\mathbb{R}^k)$  to
\begin{equation} \label{eq: FBSDE}
    \begin{cases}
&\mathrm{d}\hat{X}_t = b^{\mathfrak{u}}\bigl(t,\hat{X}_t, \hat{Y}_t, \hat{Z}_t, \hat{\Lambda}_t, (\hat{\mathbf{P}}^{u}_{t})_{u }\bigr)\mathrm{d}t+ \sigma^{\mathfrak{u}}\bigl(t,\hat{X}_t, \hat{Y}_t, \hat{Z}_t, \hat{\Lambda}_t, (\hat{\mathbf{P}}^{u}_{t})_{u}\bigr)\mathrm{d}\hat{B}_t \\
&-\mathrm{d}\hat{Y}_t = f^{\mathfrak{u}}\bigl(t,\hat{X}_t, \hat{Y}_t, \hat{Z}_t, \hat{\Lambda}_t, (\hat{\mathbf{P}}^{u}_{t})_{u}\bigr)\mathrm{d}t -\hat{Z}_t \mathrm{d}\hat{B}_t\\
\end{cases}
\end{equation}
with initial condition $\hat{X}_0 = \hat{\chi}_0$ and terminal condition 
$\hat{Y}_T = G^{\mathfrak{u}}(\hat{X}_T,(\hat{\mathbf{P}}^{u}_{T,1})_{u})$ such that (i) each $(\hat{\mathbf{P}}^u_t)_u$ is a Markov kernel characterising the conditional law of $(\hat{X}_t, \hat{Y}_t, \hat{Z}_t, \hat{\Lambda}_t)$ given $\mathfrak{u}$ for $t\in[0,T]$, (ii) we have $(\hat{X}, \hat{Y}, \hat{Z}, \hat{\Lambda}):\hat{\Omega} \rightarrow S_1$ with
\begin{equation}\label{eq:marg_laws_equal_project}
(x_t,y_t,z_t,\lambda_t)\circ (\hat{X}, \hat{Y}, \hat{Z}, \hat{\Lambda}) = (\hat{X}_t, \hat{Y}_t, \hat{Z}_t, \hat{\Lambda}_t)
\end{equation}
for $t\in[0,T]$, and (iii) we have
\begin{equation}\label{eq:Lambda_BM_expression}
\int_0^T \hat{\mathbb{E}}[|\hat{\Lambda}_t-\Lambda^{\mathfrak{u}}(t,\hat{\chi}_0, \hat{B}_{\cdot\land t})|^2]\text{d}t =0.
\end{equation}
As usual, $\hat{\mathbf{P}}^{u}_{T,1} $ is the projection of $\hat{\mathbf{P}}^{u}_{T}$ on the $\hat{X}$ component.

The auxiliary FSBDE \eqref{eq: FBSDE} will allow us to avoid dealing directly with a system of FBSDEs across the types, so we can utilise existing results and get around measurability issues in the index variable $u$. The main subject of this section is to connect this auxiliary FBSDE back to the original problem \eqref{eq:3.2}. We stress that these arguments carry over to other structural assumptions on the coefficients as long as they provide well-posedness of \eqref{eq: FBSDE}.

\begin{lemma}[Auxiliary FBSDE well-posedness]\label{Lm: FBSDE paper} 
Denote $\bar{\lambda}_1 = \lambda - 2\lambda_1 - C_1^{-1} \rho_1 - C_2^{-1} \rho_2 -(2+ C_3^{-1}  + C_4^{-1})\rho_{3} - w_1^2 - w_{4}^2$, and $ \bar{\lambda}_2 = -\lambda - 2\lambda_2 - K_1^{-1} \mu_1 - K_2^{-1}(\mu_2 + \mu_{3}) - (2+K_3^{-1}) \mu_{3}$, where $\lambda \in \mathbb{R}$,  $C_i >0, i = 1,2,3,4$ and $K_j >0 , j = 1,2,3$ are constants are to be specified. If
\[
2(\lambda_1 + \lambda_2) < -2\rho_{3} - w_1^2 - w_{4}^2 - (\mu_2 + \mu_{3})^2 - 2\mu_{3},
\]
then we can choose appropriate $\lambda$, $K_2$, and sufficiently large $C_1$, $C_2$, $C_3$, $C_4$, $K_1$, $K_3$ such that
\[
1 - K_2 \mu_2 - K_2 \mu_{3} > 0 \quad \text{and} \quad \bar{\lambda}_1, \bar{\lambda}_2 > 0.
\]
When this holds, there exists a positive constant \( \theta = 1/ [( \frac{1}{\bar{\lambda}_2} + \frac{1}{1 - K_2 \mu_2 - K_2 \mu_{3}})
( \rho^2_4 + \rho_5^2 + \frac{K_1 \mu_1 + K_3 \mu_{3}}{\bar{\lambda}_1})] \), which does not depend on T, such that if 
\[
C_1\rho_1+w_2^2+ C_3\rho_{3}+w_{4}^2 \vee C_2\rho_2+ w_3^2+ C_4\rho_{3}+w_{4}^2 \in [0, \theta),
\] 
then the FBSDE \eqref{eq: FBSDE} admits a unique solution $(\hat{X}, \hat{Y}, \hat{Z}, \hat{\Lambda}) \in \hat{\mathcal{H}}^2(\mathbb{R}^n \times \mathbb{R}^l \times \mathbb{R}^{l \times d}\times\mathbb{R}^k)$.
\end{lemma}
\begin{proof}
   By the joint measurability of the map $\Lambda$, we can start by defining the $\hat{\mathbb{F}}$-progressive process $\hat{\Lambda}_t :=\Lambda^{\mathfrak{u}}(t,\hat{\chi}_0, \hat{B}_{\cdot \land t})$. We note that we shall slightly modify it later in the proof. 
   
   Write $\hat{\mathbb{Q}}_t=\mathrm{Law}(\hat{X}_t, \hat{Y}_t, \hat{Z}_t, \hat{\Lambda}_t,\mathfrak{u})$ and $\hat{\Theta}_t=(t,\hat{X}_t, \hat{Y}_t, \hat{Z}_t, \hat{\Lambda}_t)$. Then we can express the system in terms of the coefficients
  \[
   \hat{\phi}^\mathfrak{u}(\hat{\Theta}_t,\hat{\mathbb{Q}}_t):= \phi^\mathfrak{u}(\hat{\Theta}_t,(\hat{\mathbf{P}}^{u}_t)_u),
   \]
   for $\phi=b,\sigma,f,G$, where $(\hat{\mathbf{P}}^{u}_t)_u$ is the Markov kernel for the conditional law of $(\hat{X}_t, \hat{Y}_t, \hat{Z}_t, \hat{\Lambda}_t)$ given $\mathfrak{u}$. Consider any two laws $\mathbb{Q}^i_t=\mathrm{Law}(\hat{X}^i_t, \hat{Y}^i_t, \hat{Z}^i_t, \hat{\Lambda}_t,\mathfrak{u})$ for $i=1,2$. By Assumption \ref{Assumption 4.1}, we have that the coefficients $\hat{\phi}$ are Lipschitz with respect to the particular distance
   \[
   \hat{W}_{2,m}(\mathbb{Q}^1_t,\mathbb{Q}^2_t):= W_{2,m}\bigl((\hat{\mathbf{P}}^{1,u}_t)_u,(\hat{\mathbf{P}}^{2,u}_t)_u\bigr),
   \]
   where $(\hat{\mathbf{P}}^{i,u}_t)_u$ is the disintegration of $\mathbb{Q}^i_t$ with respect to $\mathfrak{u}$, and we have
  \[
   \hat{W}^2_{2,m}(\mathbb{Q}^1_t,\mathbb{Q}^2_t) \leq \hat{\mathbb{E}}[|\hat{X}^1_t-\hat{X}^2_t|^2] + \hat{\mathbb{E}}[|\hat{Y}^1_t-\hat{Y}^2_t|^2]  + \hat{\mathbb{E}}[|\hat{Z}^1_t-\hat{Z}^2_t|^2].
   \]
Furthermore, for $\phi = b,\sigma, f$, we have $\phi^{\mathfrak{u}}(t, 0,0,0,\hat{\Lambda}_t, (\xi^u)_u) = L(\mathfrak{u},\hat{\chi}_0,\hat{B}_{\cdot \wedge t}) $, where $\xi^u \circ \pi^{-1}_{1,2,3}$ is the Dirac measure at 0 for each $u \in U$, $(\xi^u)_u \circ \pi_4^{-1}$ is the conditional law of $\hat{\Lambda}_t$ given $\mathfrak{u}$ and $L$ is a measurable function. Conditional on $\mathfrak{u} =u$, this has the same law as $L(u, \chi^u_0,B^u_{\cdot \wedge t}).$ By Assumption \ref{Assumption 4.1} and Lemma \ref{lm:joint_measureable}, we can show
   $$\int_0^T \mathbb{E}[\phi^{\mathfrak{u}}(t, 0,0,0,\hat{\Lambda}_t, (\xi^u)_u)] \mathrm{d}t < \infty.$$  
  Analogously, we have
   $$\int_0^T \mathbb{E}[G^{\mathfrak{u}}( 0, (\delta^{\tilde{u}}_0)_{\tilde{u}})] \mathrm{d}t < \infty,$$ 
   where $\delta^{u}_0$ is the Dirac measure at 0 for every $u \in U$.
   In view of this, we can adapt the arguments in the proof of \cite[Theorem 3.1]{CHEN2023105550} to conclude that there is a unique solution $(\hat{X}_t, \hat{Y}_t, \hat{Z}_t, \hat{\Lambda}_t)$ to the mean-field FSBDE  \eqref{eq: FBSDE} expressed in terms of the coefficients $\hat{\phi}$ and $\hat{\mathbb{Q}}_t=\mathrm{Law}(\hat{X}_t, \hat{Y}_t, \hat{Z}_t, \hat{\Lambda}_t,\mathfrak{u})$. 
   
   It remains to deduce from this that we have a unique solution to the actual formulation of \eqref{eq: FBSDE} which satisfies the stated conditions. To this end, let $\hat{H} \in \{\hat{Z},\hat{\Lambda}\}$. Since $\mathbb{E}[\int_0^T H_s^2 \text{d}s]<\infty$, we have that $A:=\{\omega:\int_0^T \hat{H}_s(\omega)^2 \text{d}s<\infty\}$ satisfies $\hat{\mathbb{P}}(A)=1$ with $A\in \hat{\mathcal{F}}_t$ for all $t\in[0,T]$ by completeness. Replacing $\hat{H}$ with $\hat{H}\mathbf{1}_A$ and proceeding exactly as in \eqref{eq:L2_time_project}, we obtain a predictable process $\bar{H}$ such that $\bar{H}_t(\omega)=(\hat{H}_t\mathbf{1}_A)(\omega)$ for a.e.~$t\in[0,T]$, for all $\omega \in \Omega$. By Fubini's theorem, $\mathrm{Law}(\hat{X}_t,\hat{Y}_t,\hat{Z}_t,\hat{\Lambda}_t)=\mathrm{Law}(\hat{X}_t,\hat{Y}_t,\bar{Z}_t,\bar{\Lambda}_t)$ for a.e.~$t\in[0,T]$. Thus, the stochastic integrals against $\hat{B}$ are left unaltered (indistinguishable) by changing $\hat{H}$ to $\bar{H}$ and correspondingly for the laws. Changing things accordingly, but continuing to denote the modified $\bar{Z}$ and $\bar{\Lambda}_t$ as before, we then still satisfy the FBSDE in terms of $\hat{\mathbb{Q}}_t=\mathrm{Law}(\hat{X}_t,\hat{Y}_t,\hat{Z}_t,\hat{\Lambda}_t,\mathfrak{u})$, we have that \eqref{eq:Lambda_BM_expression} holds by construction, and also \eqref{eq:marg_laws_equal_project} is satisfied by construction (after also setting $\hat{X}$ and $\hat{Y}$ to zero on a null set if necessary). Finally, letting each $(\hat{\mathbf{P}}_t^u)_u$ denote the Markov kernel obtained by disintegrating $\hat{\mathbb{Q}}_t$ with respect to $\mathfrak{u}$, and re-expressing things in terms of these kernels, the proof is complete.
\end{proof}

\subsection{Existence and uniqueness}
The above readily gives us a solution to the original FBSDE system \eqref{eq:3.2} in a certain weak sense and without asserting anything about independence of the Brownian drivers. This a legitimate notion of solution in its own right, but we shall subsequently link it back to the case of a continuum of independent Brownian drivers.

\begin{theorem}[Weak FBSDE system]\label{prop:Big-to-smal_SDE} Suppose there is a solution to \eqref{eq: FBSDE}.
Then, there exists $U_0\in\mathcal{B}(U)$ with $m(U_0)=1$ so that: for each $u\in U_0$, there is a probability measure $\bar{\mathbb{P}}^u$ on $\bar{\mathcal{F}}$ and an $S_1$-valued quadruple of progressive processes $(\hat{X}^u, \hat{Y}^u,\hat{Z}^u,\hat{\Lambda}^u)$
which satisfies the FBSDE
	\begin{equation} \label{eq: FBSDE_u}
	\begin{cases}
		&\mathrm{d}\hat{X}^u_t = b^u\bigl(t,\hat{X}^u_t, \hat{Y}^u_t, \hat{Z}^u_t, \hat{\Lambda}^u_t, (\hat{\mathbb{P}}^{\tilde{u}}_{t})_{\tilde{u}}\bigr)\mathrm{d}t+ \sigma^u\bigl(t,\hat{X}^u_t, \hat{Y}^u_t, \hat{Z}^u_t, \hat{\Lambda}^u_t, (\hat{\mathbb{P}}^{\tilde{u}}_{t})_{\tilde{u}}\bigr)\mathrm{d}\hat{B}^u_t \\
		&-\mathrm{d}\hat{Y}^u_t = f^u\bigl(t,\hat{X}^u_t, \hat{Y}^u_t, \hat{Z}^u_t, \hat{\Lambda}^u_t, (\hat{\mathbb{P}}^{\tilde{u}}_{t})_{\tilde{u}}\bigr)\mathrm{d}t -\hat{Z}^u_t \mathrm{d}\hat{B}^u_t\\
	\end{cases}
\end{equation}
	on $(\bar{\Omega},\bar{\mathcal{F}},\bar{\mathbb{P}}^u)$ with initial condition $\hat{X}^u_0 = \hat{\chi}^u_0$ and terminal condition
$\hat{Y}^u_T = G^u(\hat{X}^u_T,(\hat{\mathbb{P}}^{\tilde{u}}_{T,1})_{\tilde{u}})$ such that (i) 
\[
\hat{\mathbb{P}}^u_t = \mathrm{Law}(\hat{X}^u_t, \hat{Y}^u_t, \hat{Z}^u_t, \hat{\Lambda}^u_t)
\]
for $t\in[0,T]$, (ii)
\[
(x_t,y_t,z_t,\lambda_t)\circ (\hat{X}^u, \hat{Y}^u,\hat{Z}^u,\hat{\Lambda}^u) = (\hat{X}^u_t, \hat{Y}^u_t,\hat{Z}^u_t,\hat{\Lambda}^u_t)
\]
for $t\in[0,T]$, (iii)
\[
\hat{\mathbb{P}}^u:= \bar{\mathbb{P}}^u\circ (\hat{X}^u, \hat{Y}^u,\hat{Z}^u,\hat{\Lambda}^u)^{-1} 
\]
is a Markov kernel for $\mathcal{U}$, and (iv) 
\[
\int_0^T\mathbb{E}^{\bar{\mathbb{P}}^u}[|\hat{\Lambda}^u_s-\Lambda^u(s,\hat{\chi}_0^u,
\hat{B}^u_{\cdot \land s})|^2]\text{d}s=0.
\]
Here, $\hat{B}^u$ is a Brownian motion and $\hat{\chi}^u_0$ is an independent random variable distributed according to $ \mathrm{Law}(\chi_0^u)$ on $(\bar{\Omega},\bar{\mathcal{F}},\bar{\mathbb{P}}^u)$. Moreover, we have that $(\hat{X}^u, \hat{Y}^u,\hat{Z}^u,\hat{\Lambda}^u)_{u\in U_0}$ is in $\mathcal{S}^2_m \times \mathcal{H}^2_m$.
	\end{theorem}

\begin{proof}
	Let $(\hat{X}, \hat{Y}, \hat{Z}, \hat{\Lambda})$ be a solution to \eqref{eq: FBSDE} on $(\hat{\Omega},\hat{\mathcal{F}},\hat{\mathbb{P}})$ with Markov kernel $(\hat{\mathbf{P}}^u_t)_u$. Consider the probabilistic setup from the start of the section, and set $\bar{\Omega}^\prime = \bar{\Omega}\times U$ and $\bar{\mathcal{F}}^\prime=\mathcal{B}(\bar{\Omega}^\prime)$, which we recall yields a standard Borel space.
    
     Now consider the induced probability measure $\bar{\mathbb{P}}^\prime$ on $(\bar{\Omega}^\prime,\bar{\mathcal{F}}^\prime)$ given by
    \begin{equation}\label{eq:P_prime_Law}
    \bar{\mathbb{P}}^\prime:= \hat{\mathbb{P}} \circ (\hat{X},\hat{Y},\hat{Z},\hat{\Lambda},\hat{B},\hat{\chi},\mathbf{u})^{-1}. 
    \end{equation}
Write $\Theta_{t}=(t,x_t, y_t, z_t, \lambda_t)$. From the corresponding properties of $(\hat{X},\hat{Y},\hat{Z},\hat{\Lambda})$, we get
\[
\hat{\mathbb{E}}^\prime[\int_0^T|z_s|^2 +| \sigma^u(\Theta_s,(\hat{\mathbf{P}}^{\tilde{u}}_{s})_{\tilde{u}})|^2\mathrm{d}s] <\infty \quad\text{and}\quad \int_0^T\bar{\mathbb{E}}^\prime[|\lambda_s-\Lambda^u(s,\chi,w_{\cdot \land s})|^2]\text{d}s=0.
\]
Now define a filtration $\bar{\mathbb{F}}^\prime=(\bar{\mathcal{F}}^\prime_t)_{t\in[0,T]}$ by
\begin{equation}\label{eq:path_space_filtration}
\bar{\mathcal{F}}^\prime_t=\sigma(x_s,y_s,z_s,\lambda_s,w_s,\chi,u :s\leq t ),\quad t\in[0,T],
\end{equation}
augmented with the $\bar{\mathbb{P}}^\prime$-null sets, where the generating processes are the canonical processes defined at the start of the section. Recall that they are predictable for this filtration. Thus, the stochastic integrals $\int z_s \text{d}w_s$ and $\int \sigma^u(\Theta_s,(\hat{\mathbf{P}}^{\tilde{u}}_{s})) \text{d}w_s$ are well-defined. Moreover, we see that $w$ is a Brownian motion and $(\chi,u)$ has the correct law. Therefore, we can define 
\begin{equation} \label{eq:prob1_FBSDE}
   \begin{cases}
&\mathfrak{f}(t) := x_t - \chi - \int_0^t b^u\bigl(\Theta_s, (\hat{\mathbf{P}}^{\tilde{u}}_{s})_{\tilde{u}}\bigr)\mathrm{d}s - \int_0^t\sigma^u\bigl(\Theta_s, (\hat{\mathbf{P}}^{\tilde{u}}_{s})_{\tilde{u}}\bigr)\mathrm{d}w_s \\
 &\mathfrak{b}(t) := y_t - G\bigl(x_T,( \hat{\mathbf{P}}^{\tilde{u}}_{T,1})_{\tilde{u}}\bigr) -  \int_t^T f^u\bigl(\Theta_s, (\hat{\mathbf{P}}^{\tilde{u}}_{s})_{\tilde{u}}\bigr)\mathrm{d}s + \int_t^T z_s \mathrm{d}w_s.\\
\end{cases}   
\end{equation}
If $\bar{\mathbb{P}}^\prime(\mathfrak{f}(t)=\mathfrak{b}(t)=0)=1$ for each $t\in [0,T]$, then (by continuity of $\mathfrak{f}$ and $\mathfrak{b}$) we have that $(x_t,y_t,z_t,\lambda_t)_{t\in[0,T]}$ solves \eqref{eq: FBSDE} on $(\bar{\Omega}^\prime,\bar{\mathcal{F}}^\prime,\bar{\mathbb{F}}^\prime, \bar{\mathbb{P}}^\prime)$. To establish this, we approximate the integrands in \eqref{eq:prob1_FBSDE} by suitable simple integrands in probability. The form of the filtration $\bar{\mathbb{F}}^\prime$ gives that the terms may be taken to be functions of $(\Theta_\cdot,w_\cdot,\chi,u)$, and one can then confirm that, evaluated at $(\Theta,\hat{B},\hat{\chi},\mathfrak{u})$, we obtain simple integrands for which the corresponding integrals converge in probability to those in \eqref{eq: FBSDE}. For fixed $t$, $\mathfrak{f}^n(t)$ and $\mathfrak{b}^n(t)$ defined for the simple integrals are equal in law to the analogous expressions for \eqref{eq: FBSDE}, by the definition of $\bar{\mathbb{P}}^\prime$. Thus, their limits in probability as $n\rightarrow \infty$ must have the same law, and so we conclude that $\bar{\mathbb{P}}^\prime(\mathfrak{f}(t)=\mathfrak{b}(t)=0)=1$ from the corresponding fact for \eqref{eq: FBSDE}.

In the above, $u$ remains a random variable $(\Theta_\cdot,w_\cdot,\chi,u) \mapsto u$ on $\bar{\Omega}^\prime$, so what we have established so far is the path space equivalent of \eqref{eq: FBSDE}. To get towards \eqref{eq: FBSDE_u}, we disintegrate $\bar{\mathbb{P}}^\prime$ with respect to $(\Theta_\cdot,w_\cdot,\chi,u) \mapsto u$. By \cite[Corollary 10.4.10]{bogachev}, this gives us a Markov kernel $(\bar{\mathbb{P}}^u)_u$ for $\mathcal{B}(U)$  such that
 \begin{equation}\label{eq:disint}
\bar{\mathbb{P}}^\prime \bigl( \{(x_\cdot,y_\cdot,z_\cdot,w_\cdot,\chi)\in A\}\cap\{ u\in D\} \bigr) = \int_D \bar{\mathbb{P}}^u(A\!\times \!\{u\}) \text{d} m(u),\quad A\in \bar{\mathcal{F}},\; D\in \mathcal{B}(U),
\end{equation}
where each $\bar{\mathbb{P}}^u$ is supported on the fibre $\bar{\Omega}\times \{u\}\cong \bar{\Omega}$, so we shall view it as a probability measure on $\bar{\mathcal{F}}$. It follows immediately that there is a set $U_0\in \mathcal{B}(U)$ with $m(U_0)=1$ so that, for $u \in U_0$, $\bar{\mathbb{P}}^u(\mathfrak{f}(t)=\mathfrak{b}(t)=0)=1$ for all $t\in[0,T]$ and $\int_0^T\mathbb{E}^{\bar{\mathbb{P}}^u}[|\lambda_s-\Lambda^u(s,\chi,w_{[0,s]})|^2]\text{d}s=1$. Likewise, we can take $U_0$ so that, for $u\in U_0$, $w_\cdot$ is a Brownian motion on $(\bar{\Omega},\bar{\mathcal{F}},\bar{\mathbb{\mathbb{F}}},\bar{\mathbb{P}}^u)$, $\chi\sim \mathrm{Law}(\chi_0^u)$, and the stochastic integrals $\int z_s \text{d}w_s$ and $\int \sigma^u(\Theta_{s},(\hat{\mathbf{P}}^{\tilde{u}}_{s})_{\tilde{u}})\text{d}w_s$ are well-defined on that space, where $u$ is now treated as fixed. It remains to argue that, on $(\bar{\Omega}^\prime,\bar{\mathcal{F}}^\prime,\bar{\mathbb{P}}^u)$, these integrals agree with those in \eqref{eq:prob1_FBSDE}. Then, $(x_t,y_t,z_t,\lambda_t)_{t\in[0,T]}$ is the desired solution to \eqref{eq: FBSDE_u} on $(\bar{\Omega},\bar{\mathcal{F}},\bar{\mathbb{\mathbb{F}}},\bar{\mathbb{P}}^u)$. The approach is again one of approximation. Write $\int H_s\text{d}w_s$ for either of
$\int z_s \text{d}w_s$ or $\int \sigma^u(\Theta_{s},(\hat{\mathbf{P}}^{\tilde{u}}_{s})_{\tilde{u}})\text{d}w_s$ on $(\bar{\Omega}^\prime,\bar{\mathcal{F}}^\prime,\bar{\mathbb{F}}^\prime, \bar{\mathbb{P}}^\prime)$, and assume $z$ and $\sigma$ are bounded ($\bar{\mathbb{P}}^\prime$-a.s.~for a.e.~$t$), as we may otherwise apply a cut-off and another analogous limiting argument.
Then, we can find simple integrands $H^{n}$ that are $\bar{\mathbb{P}}^\prime\otimes \mathrm{Leb}$-a.e.~convergent and bounded such that $\int_0^t H^n_s\text{d}w_s$ converges $\bar{\mathbb{P}}^\prime$-a.s.~to $\int_0^t H_s 
\text{d}w_s$. By \eqref{eq:disint}, we can take $U_0$ to be such that the same is true with respect to $\bar{\mathbb{P}}^u$ for each $u\in U_0$. Since $\bar{\mathbb{P}}^u$ is supported on $\bar{\Omega}\times \{u\}$, we deduce that $\int_0^T\mathbb{E}^{\bar{\mathbb{P}}^u}[|H^n_s-H^u_s|^2]\text{d}s$ tends to zero, where $H^u$ stands for $H$ with $u$ held fixed. Thus, the simple integrals $\int_0^t H^n_s 
\text{d}w_s$ tend to both $\int_0^t H_s 
\text{d}w_s$ and $\int_0^t H^u_s 
\text{d}w_s$ on $(\bar{\Omega}^\prime,\bar{\mathcal{F}}^\prime,\bar{\mathbb{F}}^\prime, \bar{\mathbb{P}}^u)$ in probability, for $t\in[0,T]$, and hence we conclude that the two are indistinguishable under $\bar{\mathbb{P}}^u$.

For a given $u\in U_0$, we may view $\bar{\mathbb{P}}^u$ as a probability measure on $(\bar{\Omega},\bar{\mathcal{F}})$, as it is supported on $\bar{\Omega}\times\{u\}$. We then set $\hat{X}_t^u(x,y,[z],[\lambda],\chi,w)= x_t$ and likewise for $\hat{Y}^u$, $\hat{Z}^u$, $\hat{\Lambda}^u$, $\hat{\chi}_0^u$, and $\hat{B}^u$, noting that we may also treat these as defined on $\bar{\Omega}^\prime$ in the obvious way. Then, recalling the definition of $\bar{\mathbb{P}}^\prime$ in \eqref{eq:P_prime_Law}, and using properties (i) and (ii) of \eqref{eq: FBSDE} as well as the defining property \eqref{eq:disint} of $\bar{\mathbb{P}}^u$, it follows that, for every $t\in[0,T]$, we have 
\begin{align*}
\int_D \hat{\mathbf{P}}^{u}_{t}(O)\text{d}m(u)  &= \bar{\mathbb{P}}^\prime\bigl( \{(\hat{X}_t^u,\hat{Y}_t^u,\hat{Z}_t^u,\hat{\Lambda}_t^u)\in O\} \cap \{\mathfrak{u}\in D \}\bigr)\\ 
&= \int_D \bar{\mathbb{P}}^u(\{(\hat{X}_t^u,\hat{Y}_t^u,\hat{Z}_t^u,\hat{\Lambda}_t^u)\in O\}\times \{u\}) \text{d} m(u) \\
&= \int_D \bar{\mathbb{P}}^u \circ (\hat{X}_t^u,\hat{Y}_t^u,\hat{Z}_t^u,\hat{\Lambda}_t^u)^{-1}(O) \text{d} m(u),
\end{align*}
for all $D\in \mathcal{B}(U)$ and $O\in \mathcal{B}(\mathbb{R}^n \times \mathbb{R}^l \times \mathbb{R}^{l \times d}\times\mathbb{R}^k)$, where the last line (and the remaining part of the proof) comes back to viewing $\hat{\mathbb{P}}^u$ as a measure on $\bar{\mathcal{F}}$. Hence, we may define $\hat{\mathbb{P}}^u_t:= \bar{\mathbb{P}}^u \circ (\hat{X}_t^u,\hat{Y}_t^u,\hat{Z}_t^u,\hat{\Lambda}_t^u)^{-1}$ and
replace $(\hat{\mathbf{P}}^{u}_{t})_u$ with $(\hat{\mathbb{P}}^u_t)$ in the dynamics of the FBSDE \eqref{eq:prob1_FBSDE} and its terminal condition without affecting the fact that it is satisfied by $(\hat{X}^u,\hat{Y}^u,\hat{Z}^u,\hat{\Lambda}^u)$. This confirms that \eqref{eq: FBSDE_u} holds with the conditions (iv) and (i) both being satisfied. Finally, we note that condition (ii) holds by definition of the processes and, in turn, (iii) holds by construction of $\hat{\mathbb{P}}$. This completes the proof.
\end{proof}

\begin{corollary}[Existence of strong solutions] \label{co:existence}
Let the conditions of Lemma \ref{Lm: FBSDE paper}  be in force. Then, the system of FBSDEs \eqref{eq:3.2} admits a strong solution in the sense of Definition \ref{def:strong_soln}.
\end{corollary}
\begin{proof}
    
Let $U_0 \in \mathcal{B}(U)$ with $m(U_0)=1$ be such that we have weak solutions in the sense of Theorem \ref{prop:Big-to-smal_SDE} for every $u\in U_0$. We work on a given complete probability space $(\Omega,\mathcal{F},\mathbb{P})$, as fixed at the start of the paper, with a family of independent Brownian motions $(B^u)_u$ and initial random variables $(\chi^u_0)_u$ independent of the Brownian motions.
For a fixed $u\in U_0$, we consider the FBSDE
\begin{equation} \label{eq: FBSDE_change}
    \begin{cases}
&\mathrm{d}\tilde{X}^u_t = b^u\bigl(t,\tilde{X}^u_t, \tilde{Y}^u_t, \tilde{Z}^u_t, \tilde{\Lambda}^u_t, (\hat{\mathbb{P}}^{\tilde{u}}_{t})_{\tilde{u}}\bigr)\mathrm{d}t+ \sigma^u\bigl(t,\tilde{X}^u_t, \tilde{Y}^u_t, \tilde{Z}^u_t, \tilde{\Lambda}^u_t, (\hat{\mathbb{P}}^{\tilde{u}}_{t})_{\tilde{u}}\bigr)\mathrm{d}B^u_t \\
&-\mathrm{d}\tilde{Y}^u_t = f^u\bigl(t,\tilde{X}^u_t, \tilde{Y}^u_t, \tilde{Z}^u_t, \tilde{\Lambda}^u_t, (\hat{\mathbb{P}}^{\tilde{u}}_{t})_{\tilde{u}}\bigr)\mathrm{d}t -\tilde{Z}^u_t \mathrm{d}B^u_t\\
\end{cases}
\end{equation}
with initial condition $\tilde{X}^u_0 = \chi^u_0$ and terminal condition
$\tilde{Y}^u_T = G^u(\tilde{X}^u_T,(\hat{\mathbb{P}}^{\tilde{u}}_{T,1})_{\tilde{u}})$. We set $\tilde{\Lambda}^u_t :=\Lambda^u(t,\chi_0^u,
B^u_{\cdot \land t})$, noting that this is $\mathbb{F}^u$-progressive (recall that $\mathbb{F}^u$ is the completed filtration generated by $B^u$ and $\chi^u_0$) and we let the Markov kernels be those from Theorem \ref{prop:Big-to-smal_SDE}, so these and $\tilde{\Lambda}^u$ are given exogenously. Thus, \eqref{eq: FBSDE_change} is a standard (fully coupled) FBSDE with progressively measurable coefficient functions.  By Assumption \ref{Assumption 4.1} and the conditions of  Lemma \ref{Lm: FBSDE paper}, it follows from \cite[Theorem 3.1]{Pardoux} that \eqref{eq: FBSDE_change} has a unique strong solution $(\tilde{X}^u, \tilde{Y}^u, \tilde{Z}^u)$ for the given definition of $\tilde{\Lambda}^u$. Following the same steps as in the proof of Lemma \ref{Lm: FBSDE paper}, we re-define $\tilde{Z}^u$ and $\tilde{\Lambda}^u$ to be predictable $L^2([0,T],\mathbb{R}^{l\times d})$- and $L^2([0,T],\mathbb{R}^{k})$-valued processes, respectively, without affecting the FBSDE. We continue to denote them by $\tilde{Z}^u$ and $\tilde{\Lambda}^u$, but stress that we now only have $\int_0^T\mathbb{E}[|\tilde{\Lambda}^u_s-\Lambda^u(s,\chi_0^u,
B^u_{\cdot \land s})|^2]\text{d}s=0$.

Recall $\bar{\Omega}=S_1\times S_2$ introduced in \eqref{eq:S1S2}, let $\bar{\mathcal{F}}$ be its Borel $\sigma$-algebra, and let $\mathbb{G}=(\mathcal{G}_t)_{t\in[0,T]}$ be given by $\mathcal{G}_t=\sigma(x_s,y_s,z_s,\lambda_t,\chi,w_s:s\leq t)$. In view of the above, as in the proof of Theorem \ref{prop:Big-to-smal_SDE}, the law $ \tilde{\mu}^u=\mathbb{P}^u\circ (\tilde{X}^u, \tilde{Y}^u,\tilde{Z}^u,\tilde{\Lambda}^u,\chi_0^u,B^u)^{-1}$ is well-defined on $(\bar{\Omega},\bar{\mathcal{F}})$, and we have $(x_t,y_t,z_t,\lambda_t)\circ (\tilde{X}^u, \tilde{Y}^u,\tilde{Z}^u,\tilde{\Lambda}^u) = (\tilde{X}^u_t, \tilde{Y}^u_t,\tilde{Z}^u_t,\tilde{\Lambda}^u_t)$ for $t\in[0,T]$.

In the notation of \cite{kurtz2014weakstrongsolutionsgeneral}, we set $\nu=\mathbb{P}\circ(\chi_0^u,B^u)^{-1}=\mathrm{Law}(\chi_0^u)\otimes \mathcal{W}$, where $\mathcal{W}$ is the Wiener measure, and let $\mathcal{S}_{\Gamma,\nu}$ denote the set of joint laws $\mu\in \mathcal{P}(S_1\times S_2)$ such that (i) $\mu(S_1\times \cdot)=\nu$ and (ii) $\mu$ satisfies the following constraint $\Gamma$: $\mu(\mathfrak{f}(t)=\mathfrak{b}(t)=0)=1$ for all $t\in [0,T]$, where $\mathfrak{f}(t)$ and $\mathfrak{b}(t)$ from \eqref{eq:prob1_FBSDE} are defined with $\hat{\mathbb{P}}^{\tilde{u}}_{\cdot}$ in place of $\hat{\mathbf{P}}^{\tilde{u}}_{\cdot}$ and with the stochastic integrals constructed under $\mu$ (for the filtration $\mathbb{G}$). It follows as in the proof of Theorem \ref{prop:Big-to-smal_SDE} that the constraint $\Gamma$ is well-defined and that $\tilde{\mu}^u,\hat{\mu}^u\in \mathcal{S}_{\Gamma,\nu}$, where $\hat{\mu}^u=\mathbb{P}^u\circ (\hat{X}^u, \hat{Y}^u,\hat{Z}^u,\hat{\Lambda}^u,\hat{\chi}_0^u,\hat{B}^u)^{-1}$. By \cite[Theorem 3.1]{Pardoux}, we have strong (pathwise) uniqueness, so \cite[Theorem 1.5]{kurtz2014weakstrongsolutionsgeneral} and \cite[Lemma 2.10]{kurtz2014weakstrongsolutionsgeneral} gives that $\tilde{\mu}^u=\hat{\mu}^u$. Using property (ii) of Theorem \ref{prop:Big-to-smal_SDE} and the same for $(\tilde{X}^u, \tilde{Y}^u,\tilde{Z}^u,\tilde{\Lambda}^u)$, we can therefore conclude that
\[
\tilde{\mathbb{P}}^u =\hat{\mathbb{P}}^u \quad \text{and} \quad \tilde{\mathbb{P}}_t^u =\hat{\mathbb{P}}_t^u,
\]
for all $t\in [0,T]$, where
\begin{equation}\label{eq:tilde_path_laws}
\tilde{\mathbb{P}}^u := \mathbb{P}\circ (\tilde{X}^u, \tilde{Y}^u,\tilde{Z}^u,\tilde{\Lambda}^u)^{-1}
\end{equation}
and
\begin{equation}\label{eq:tilde_marginal_laws}
\tilde{\mathbb{P}}_t^u := \mathbb{P}\circ (\tilde{X}_t^u, \tilde{Y}_t^u,\tilde{Z}_t^u,\tilde{\Lambda}_t^u)^{-1}=\tilde{\mathbb{P}}^u\circ (x_t,y_t,z_t,\lambda_t)^{-1}.
\end{equation}
 Repeating the above for every $u\in U_0$ yields a solution $(\tilde{X}, \tilde{Y},\tilde{Z},\tilde{\Lambda}) \in \mathcal{S}^2_m \times \mathcal{H}^2_m$ to the system of FBSDEs \eqref{eq:3.2} with $U_0$ in place of $U$.

 For $u\in U\setminus U_0$, we can still construct a strong $S_1$-valued solution to \eqref{eq: FBSDE_change}, as in the above, with the Markov kernels given exogenously in terms of $(\tilde{\mathbb{P}}^u)_{u\in U_0}$. Defining the law $\tilde{\mathbb{P}}^u$ as in \eqref{eq:tilde_path_laws} for $u\in U\setminus U_0$, we then also have \eqref{eq:tilde_marginal_laws} for all $u\in U$, by construction. Of course, we cannot just assert that $(\tilde{\mathbb{P}}^u)_{u\in U}$ is a Markov kernel for $\mathcal{B}(U)$, but $m(U\setminus U_0)=0$ implies that it is at least a Markov kernel for the completion of $\mathcal{B}(U)$ with respect to $m$. Note also that the values of $\tilde{\mathbb{P}}^u$ on $U\setminus U_0$ does not affect the FBSDE dynamics and the terminal condition. Thus, we have indeed constructed a strong solution $(\tilde{X}, \tilde{Y},\tilde{Z},\tilde{\Lambda}) \in \mathcal{S}^2_m \times \mathcal{H}^2_m$ to the system \eqref{eq: FBSDE_change} in the sense of Definition \ref{def:strong_soln} and so the proof is complete.
\end{proof}

The next observation essentially confirms that we can argue in reverse in Corollary \ref{co:existence} and Theorem \ref{prop:Big-to-smal_SDE}. It also gives us a convenient way of deducing uniqueness.

\begin{prop}\label{prop:small_to_big} Let Assumption \ref{Assumption 4.1} hold and let the conditions of Lemma \ref{Lm: FBSDE paper} apply. Then, any strong solution $(X^u,Y^u,Z^u,\Lambda^u)_u$ to \eqref{eq:3.2} in the sense of Definition \ref{def:strong_soln} gives rise to a strong solution $(\hat{X},\hat{Y},\hat{Z},\hat{\Lambda})$ to \eqref{eq: FBSDE} satisfying the conditions (i)--(iii) of such a solution and for which it holds that $\hat{\mathbf{P}}^u_t =\mathbb{P}^u_t$ for all $t\in[0,T]$ and all $u\in U_0$ for a set $U_0\in\mathcal{B}(U)$ with $m(U_0)=1$.
\end{prop}
\begin{proof}
 Let $(\hat{\Omega},\hat{\mathcal{F}},\hat{\mathbb{P}})$ and $(\hat{B},\hat{\chi}_0, \mathfrak{u})$ be as in Lemma \ref{Lm: FBSDE paper}. Exactly as in that lemma (in fact, requiring only a standard well-posedness result without the mean-field aspect), we have a strong (and pathwise unique) solution to the FBSDE
\begin{equation} \label{FBSDe small_to_big}
    \begin{cases}
&\mathrm{d}\hat{X}'_t = b^{\mathfrak{u}}\bigl(t,\hat{X}'_t, \hat{Y}'_t, \hat{Z}'_t, \hat{\Lambda}^\prime_t, (\mathbb{P}^{u}_{t})_{u }\bigr)\mathrm{d}t+ \sigma^{\mathfrak{u}}\bigl(t,\hat{X}'_t, \hat{Y}'_t, \hat{Z}'_t, \hat{\Lambda}^\prime_t, (\mathbb{P}^{u}_{t})_{u}\bigr)\mathrm{d}\hat{B}_t \\
&-\mathrm{d}\hat{Y}'_t = f^{\mathfrak{u}}\bigl(t,\hat{X}'_t, \hat{Y}'_t, \hat{Z}'_t, \hat{\Lambda}^\prime_t, (\mathbb{P}^{u}_{t})_{u}\bigr)\mathrm{d}t -\hat{Z}'_t \mathrm{d}\hat{B}_t\\
\end{cases}
\end{equation}
with initial condition $\hat{X}'_0 = \hat{\chi}_0$ and terminal condition
$\hat{Y}'_T = G^{\mathfrak{u}}\bigl(\hat{X}'_T,(\mathbb{P}^{u}_{T,1})_{u}\bigr)$, where the only difference from the formulation in \eqref{eq: FBSDE}--\eqref{eq:Lambda_BM_expression} is that the Markov kernels $(\mathbb{P}^{u}_{t})_{u }$ are exogenous and correspond to a given solution $(X^u,Y^u,Z^u,\Lambda^u)_u$ to \eqref{eq:3.2}. 

Arguing as in Theorem \ref{prop:Big-to-smal_SDE}, $(\hat{X}^\prime,\hat{Y}^\prime,\hat{Z}^\prime,\hat{\Lambda}^\prime)$ induces a solution $(\hat{X}^{u,\prime},\hat{Y}^{u,\prime},\hat{Z}^{u,\prime},\hat{\Lambda}^{u,\prime})_{u\in U_0}$ to \eqref{eq: FBSDE_u} except that the Markov kernel is given exogenously by $(\tilde{\mathbb{P}}^{u}_{t})_{u }$. This holds for all $u\in U_0$, where $U_0\subset U$ is some Borel set with $m(U_0)=1$. Moreover, we get from the particular construction that $\hat{\mathbb{P}}^{u,\prime}_t := \hat{\mathbb{P}}\circ(\hat{X}^{u,\prime}_t,\hat{Y}^{u,\prime}_t,\hat{Z}^{u,\prime}_t,\hat{\Lambda}^{u,\prime}_t)^{-1}$ and $\hat{\mathbb{P}}^{u,\prime} := \hat{\mathbb{P}}\circ(\hat{X}^{u,\prime},\hat{Y}^{u,\prime},\hat{Z}^{u,\prime},\hat{\Lambda}^{u,\prime})^{-1}$ satisfy $\hat{\mathbb{P}}^{u,\prime}_t=\hat{\mathbb{P}}^{u,\prime}\circ(x_t,y_t,z_t,\lambda_t)^{-1}$. Using the property \eqref{eq:marg_laws_equal_project} of $(\hat{X}^\prime,\hat{Y}^\prime,\hat{Z}^\prime,\hat{\Lambda}^\prime)$, it also follows as in the proof of Theorem \ref{prop:Big-to-smal_SDE} that each $(\hat{\mathbb{P}}^{u,\prime}_t)_u$ can serve as the Markov kernel for the conditional law of $(\hat{X}^{\prime},\hat{Y}^{\prime},\hat{Z}^{\prime},\hat{\Lambda})$ given $\mathfrak{u}$. Now, for each $u\in U_0$, $(\hat{X}^{u,\prime},\hat{Y}^{u,\prime},\hat{Z}^{u,\prime},\hat{\Lambda}^{u,\prime})$ solves the same FBSDE as $(X^u,Y^u,Z^u,\Lambda^u)$ only on a different probability space, so we can argue exactly as in the proof of Corollary \ref{Lm: FBSDE paper} to conclude that $(\mathbb{P}^u)_{u\in U_0}=(\hat{\mathbb{P}}^{u,\prime})_{u\in U_0}$. Since $\mathbb{P}^u_t=\mathbb{P}^u\circ (x_t,y_t,z_t,\lambda_t)^{-1}$ by Definition \ref{def:strong_soln}, and since the same holds for $(\hat{X}^{u,\prime},\hat{Y}^{u,\prime},\hat{Z}^{u,\prime},\hat{\Lambda}^{u,\prime})$, we get $(\mathbb{P}^u_t)_{u\in U_0}=(\hat{\mathbb{P}}_t^{u,\prime})_{u\in U_0}$ for all $t\in[0,T]$.  Finally, we said above that the Markov kernel $(\hat{\mathbb{P}}_t^{u,\prime})_{u\in U_0}$ characterizes the conditional law of $(\hat{X}^{\prime},\hat{Y}^{\prime},\hat{Z}^{\prime},\hat{\Lambda})$ given $\mathfrak{u}$. Thus, we can indeed replace $(\mathbb{P}^u_t)_{u}$ by $(\hat{\mathbf{P}}^{u,\prime}_t)_u$ in \eqref{FBSDe small_to_big} and its terminal condition, where $(\hat{\mathbf{P}}^{u,\prime}_t)_u$ produces the conditional law of $(\hat{X}^{\prime},\hat{Y}^{\prime},\hat{Z}^{\prime},\hat{\Lambda})$ given $\mathfrak{u}$, and we have that $\hat{\mathbf{P}}^{u,\prime}_t=\mathbb{P}^u_t$ for all $t\in[0,T]$ and all $u \in U_0$ as required.
\end{proof}

\begin{corollary}[Pathwise uniqueness of strong solutions] \label{co:uniqueness} 
    Let Assumption \ref{Assumption 4.1} hold and let the conditions of Lemma \ref{Lm: FBSDE paper} apply. Then, strong solutions to the FBSDE system \eqref{eq:3.2} in the sense of Definition \ref{def:strong_soln} exhibit strong (pathwise) uniqueness for every $u\in U$.
\end{corollary}
\begin{proof}
Suppose we have two strong solutions on $(\Omega,\mathcal{F},\mathbb{F},\mathbb{P})$ in the sense of Definition \ref{def:strong_soln} with Markov kernels $(\mathbb{P}^{1,u}_t)_u$ and $(\mathbb{P}^{2,u}_t)_u$. By Proposition \ref{prop:small_to_big}, they give rise to two (strong) solutions of \eqref{eq: FBSDE}. But the uniqueness in Lemma \ref{Lm: FBSDE paper} (and the property of the Markov kernels in Proposition \ref{prop:small_to_big}) then implies that $\mathbb{P}^{1,u}_t$ and $\mathbb{P}^{2,u}_t$ agree for all $t\in[0,T]$ and all $u\in U_0$ for a set $U_0\in \mathcal{B}(U)$. Hence, the two solutions satisfy a standard FBSDE (with the Markov kernel frozen) for every $u\in U$. Then, the pathwise uniqueness from \cite[Theorem 3.1]{Pardoux} gives pathwise uniqueness of the two solutions for every $u\in U$. 
\end{proof}

The existence in Corollary \ref{co:existence} and uniqueness in Corollary \ref{co:uniqueness}, gives us Theorem \ref{Theorem 4.2}.

\begin{remark}[Yamada--Watanabe] We note that the proofs of Theorem \ref{prop:Big-to-smal_SDE} and Corollary \ref{co:existence} could readily be extended to give the statement that, for the general mean-field systems of FBSDEs that we consider here, strong (pathwise) uniqueness and the existence of a weak solution is equivalent to uniqueness in law and the existence of a strong solution.
\end{remark}

\appendix
\section{Appendix} \label{Lion}

Suppose Assumption \ref{Ass:5.2} holds. Then, by Theorem \ref{Theorem for FBSDE} and Theorem \ref{Theorem 5.4}, the system of FBSDEs \eqref{eq:4.1} under control $\hat{\boldsymbol{\alpha}}$, the system of FBSDEs \eqref{eq:4.1} under control $\hat{\boldsymbol{\alpha}}+\epsilon\boldsymbol{\pi}$, and the variational equation \eqref{eq::variation} all admit a unique strong solution. With the notation of Section \ref{section 4}, we denote the strong solutions by $(\Theta^{u,\alpha})_u := (X^{u},Y^{u},Z^{u})_u$, $(\Theta^{u,\hat{\alpha}+\epsilon\pi})_u := (X^{u,\varepsilon},Y^{u,\varepsilon},Z^{u,\varepsilon})_u$, and $(\Theta^{u,'})_u := (X^{u,'},Y^{u,'},Z^{u,'})_u$, respectively. We need to know that, when these solutions are considered together, their joint laws form a Markov kernel.

\begin{lemma}[Markov kernel] \label{Lm:Markov Kernel Lemma}
With the above definitions, the joint laws $(\mathbb{P}_{\Theta^{u,\alpha},\Theta^{u,\alpha+\epsilon\pi},\Theta^{u,'}})_u$ can be taken to be a Markov kernel for $\mathcal{U}$.
\end{lemma}
\begin{proof} 
The well-posedness results each give that the three marginal Markov kernels, corresponding to $(\Theta^{u,\alpha})_u$, $(\Theta^{u,\hat{\alpha}+\epsilon\pi})_u$, and $(\Theta^{u,'})_u$ are Markov kernels. We only sketch the proof that their joint law can be assumed to be a Markov kernel, as it amounts to repeating the steps of Proposition \ref{prop:small_to_big}. The triple $(\Theta^{u,\alpha},\Theta^{u,\hat{\alpha}+\epsilon\pi},\Theta^{u,'})_u$ solves a larger FBSDE system which only involves the marginal Markov kernels. As in the proof of Proposition \ref{prop:small_to_big}, the triple gives rise to a variant of \eqref{eq: FBSDE} from which we obtain a a variant of \eqref{eq: FBSDE_u}, where the Markov kernels in the dynamics are given exogenously by the aforementioned marginal Markov kernels, and we then deduce that the laws are the same as those of our original triple. But we know that the laws for the variant of \eqref{eq: FBSDE_u} are Markov kernels, by construction, and hence the conclusion follows.
\end{proof}

 In the following, we adopt the same notation as in Section \ref{Lion construct}. The next result provides two Lipschitz properties for partial $L_m$-derivatives. 

\begin{lemma}[$L_m$ Lipschitzness]\label{lemma:lipschitz}
Denote $
\bar{\Theta}^{u}_t = \bigl(X^u_t, Y^u_t, Z^u_t, \alpha^u_t, (\mathbb{P}^{\tilde{u}}_{t})_{\tilde{u}}\bigr)$, $ \bar{\vartheta}^{u}_t = (X^u_t, Y^u_t, Z^u_t, \alpha^u_t)$ and $(\tilde{\bar{\Theta}}^{u}_t,\tilde{\bar{\vartheta}}^{u}_t)$ is an independent copy of $(\bar{\Theta}^{u}_t,\bar{\vartheta}^{u}_t)$. Let random variables $\mathfrak{X}^1, \mathfrak{X}^2:\Omega \rightarrow \mathbb{R}^n$ be given, and let $\tilde{\mathfrak{X}}^1$ and $\tilde{\mathfrak{X}}^2$ be independent copies of these such that the joint law of each $(\tilde{\mathfrak{X}}^{i,u},\tilde{\bar{\vartheta}}^{u}_t)$ is the same as that of $(\mathfrak{X}^{i,u},\bar{\vartheta}^{u}_t),\quad i = 1,2$. Let $\eta_1, \eta_2 \in \mathcal{P}^2_m( \mathbb{R}^n \times \mathbb{R}^l \times \mathbb{R}^{l \times d} \times \mathbb{R}^k ) $ with marginals $\eta_j^{(i)}$ for $j = 1,2$ and $i = 1,2,3,4$. Suppose, without loss of generality, that $\eta_1, \eta_2$ only differ in the first marginal.
We assume that the function $\gamma$ is $L_m$ differentiable in the sense of Definition \ref{def:Lm-derivative} and satisfies the following condition, for $u \in U$, all $(t, x, y, z, a) \in [0,T] \times \mathbb{R}^n \times \mathbb{R}^l \times \mathbb{R}^{l \times d} \times \mathbb{R}^k$,
\begin{equation} \label{ineq: marginal_distance}
    \bigl| \gamma^u(t, x, y, z, a, \eta_1) - \gamma^u(t, x, y, z, a, \eta_2) \bigr| 
\le \rho  W_{2,m}\big( \eta_1^{(1)}, \eta_2^{(1)} \big),
\end{equation}
then
\begin{equation}\label{eq:lipschitz1}
	    |
	\int \tilde{\mathbb{E}}\big[ \partial_{P^{(1)}}\gamma^u(t, \bar{\Theta}^{u}_t)(\tilde{\bar{\vartheta}}^{\tilde{u}}_t) \tilde{\mathfrak{X}}_t^{1,\tilde{u}} \big] \mathrm{d}m(\tilde{u})-\int \tilde{\mathbb{E}}\big[ \partial_{P^{(1)}}\gamma^u(t, \bar{\Theta}^{u}_t)(\tilde{\bar{\vartheta}}^{\tilde{u}}_t) \tilde{\mathfrak{X}}_t^{2,\tilde{u}} \big] \mathrm{d}m(\tilde{u})|
    \le \rho  W_{2,m}(\mathbb{P}_{\mathfrak{X}^1}, \mathbb{P}_{\mathfrak{X}^2}).
	\end{equation}
Likewise, given $\mathfrak{q}^1$ and $\mathfrak{q}^2$, and let $\tilde{\mathfrak{q}}^1$ and $\tilde{\mathfrak{q}}^2$ be independent copies of these such that the joint law of each $(\tilde{\mathfrak{q}}^{i,u},\tilde{\bar{\Theta}}^{u})$ is the same as that of $(\mathfrak{q}^{i,u},\bar{\Theta}^{u}),\quad i = 1,2$. If we additionally assume the family of functions $(\gamma^u)_u$ satisfies
\begin{equation}\label{eq:L2_assump_Lions}
[\int_U \tilde{\mathbb{E}}  [|\partial_{P^{(1)}}\gamma^{\top,\tilde{u}}(t,\tilde{\bar{\Theta}}^{\tilde{u}}_t)(\bar{\vartheta}^{u}_t)|^2]\mathrm{d}m(\tilde{u})]^{\frac{1}{2}} <C',
\end{equation}
for $u\in U$.
Then
\begin{equation}\label{eq:lipschitz2}
	|\int \tilde{\mathbb{E}}  [ \partial_{P^{(1)}}\gamma^{\top,\tilde{u}}(t, \tilde{\bar{\Theta}}^{\tilde{u}}_t)(\bar{\vartheta}^{u}_t)\tilde{\mathfrak{q}}^{1,\tilde{u}}_t  ] \mathrm{d}m(\tilde{u})-\int \tilde{\mathbb{E}}  [ \partial_{P^{(1)}}\gamma^{\top,\tilde{u}}(t, \tilde{\bar{\Theta}}^{\tilde{u}}_t)(\bar{\vartheta}^{u}_t)\tilde{\mathfrak{q}}^{2,\tilde{u}}_t  ]\mathrm{d}m(\tilde{u})|\le  C'   W_{2,m}(\mathbb{P}_{\mathfrak{q}^1}, \mathbb{P}_{\mathfrak{q}^2}).
	\end{equation}
\end{lemma}

\begin{proof} By Cauchy-Schwarz, the left-hand side of \eqref{eq:lipschitz1} is bounded by
    \begin{align}
&| \Big[
   \int \tilde{\mathbb{E}}\big[ \partial_{P^{(1)}}\gamma^u(t, \bar{\Theta}^{u}_t)(\tilde{\bar{\vartheta}}^{\tilde{u}}_t)   ( \tilde{\mathfrak{X}}_t^{1,\tilde{u}}- \tilde{\mathfrak{X}}_t^{2,\tilde{u}}  ) \big] \mathrm{d}m(\tilde{u}) \Big]|  
\notag \\
   & \leq \Big[
   \int \tilde{\mathbb{E}}\big[ |\partial_{P^{(1)}}\gamma^u(t, \bar{\Theta}^{u}_t)(\tilde{\bar{\vartheta}}^{\tilde{u}}_t)|^2 \big] \mathrm{d}m(\tilde{u}) \Big]^{\frac{1}{2} }   ( \int \mathbb{E}[ |\tilde{\mathfrak{X}}_t^{1,\tilde{u}}- \tilde{\mathfrak{X}}_t^{2,\tilde{u}}|]^2\mathrm{d}m(\tilde{u})  )^{\frac{1}{2} }.\label{eq:lip_wasserstein}
\end{align}
Let $\Theta^j := (x,y,z,\alpha, \eta_j) , j = 1,2 $. Take a sequence $\Xi_n$ with $\int_U||\Xi^{u}_n||_2 \mathrm{d}m(u) \leq 1 $ such that
\begin{equation}\label{eq:uniform_L2_Lions}
\int_U \mathbb{{E}}\bigl[|\partial_{P^{(1)}}\gamma^u(\Theta^{1})(\Xi^{1,\tilde{u}})|^2 \bigr]^{\frac{1}{2}}\mathrm{d}m(\tilde{u}) = \lim_{n\rightarrow \infty} \bigl| \int_U \mathbb{{E}}\bigl[\partial_{P^{(1)}}\gamma^u(\Theta^{1})(\Xi^{1,\tilde{u}})\cdot \Xi^{\tilde{u}}_n\bigr]\mathrm{d}m(\tilde{u})\bigr|
\end{equation}
For any given $n\geq 1$, we can take $\eta_2$ such that $\Xi^1-\Xi^2 = \Xi_n$. Now \eqref{eq:generalised_L-diff} and the Lipschitz assumption on $b$ gives that 
\begin{align*}
&\bigl| \int_U\mathbb{{E}}\bigl[\partial_{P^{(1)}}\gamma^u(\Theta^{1})(\Xi^{1,\tilde{u}})\cdot \Xi^{\tilde{u}}_n\bigr]\mathrm{d}m(\tilde{u})\bigr| \leq \bigl| \gamma^{u}(\Theta^2)-\gamma^{u}(\Theta^1) \bigr| + o(1)  \\
&\leq \rho W_{2,m}(\eta^{(1)}_1, \eta^{(1)}_2) + o(1)
\leq \rho \int_U||\Xi^{1,u}-\Xi^{2,u}||_2 \mathrm{d}m(u) + o(1) \leq \rho + o(1)
\end{align*}
as $n\rightarrow\infty$. Thus, the left-hand side of \eqref{eq:uniform_L2_Lions} is bounded by $\rho$, as desired. For any coupling $\mu$ of $\mathbb{P}_{\mathfrak{X}^1}$ and $\mathbb{P}_{\mathfrak{X}^2}$, we can consider some probability space $(\tilde{\Omega},\tilde{\mathcal{F}},\tilde{\mathbb{P}})$ with random variables $(\tilde{\mathfrak{X}}^1,\tilde{\mathfrak{X}}^2, (\tilde{\bar{\vartheta}}^u)_u)$ so that the joint law of each $(\tilde{\mathfrak{X}}^{i,u},\tilde{\bar{\vartheta}}^{u})$ is as required, the law of each $\tilde{\mathfrak{X}}^{i}$ is $\mathbb{P}_{\mathfrak{X}^i}$, and the joint law of $(\tilde{\mathfrak{X}}^1,\tilde{\mathfrak{X}}^2)$ is $\mu$. Placing things on the product space, these are also independent copies as required. As any coupling can be obtained in this way, taking an infimum over all such constructions on the right-hand side of \eqref{eq:lip_wasserstein} returns $\rho W_{2,m}(\mathbb{P}_{\mathfrak{X}^1}, \mathbb{P}_{\mathfrak{X}^2})$. Finally, we can note that the left-hand side of \eqref{eq:lipschitz1} is the same for all these constructions (as the joint law of $(\tilde{\bar{\vartheta}}^u)_u$ with $\tilde{\mathfrak{X}
}^1$ or $\tilde{\mathfrak{X}}^2$ is kept fixed). Thus, we obtain \eqref{eq:lipschitz1} which proves the first part of the statement. The argument for \eqref{eq:lipschitz2} is analogous, given the assumption \eqref{eq:L2_assump_Lions}.
\end{proof}

\bibliographystyle{siam}

\bibliography{refs}	

\end{document}